\documentclass[pdflatex,sn-mathphys-ay]{sn-jnl}

\usepackage{graphicx}%
\usepackage{multirow}%
\usepackage{amssymb}
\usepackage{amsmath}
\usepackage{amsfonts}
\usepackage{mathrsfs}
\usepackage{mathtools}
\usepackage[title]{appendix}%
\usepackage{xcolor}%
\usepackage{textcomp}%
\usepackage{manyfoot}%
\usepackage{booktabs}%
\usepackage{algorithm}%
\usepackage{algorithmicx}%
\usepackage{algpseudocode}%
\usepackage{listings}%
\usepackage{enumitem}
\usepackage{tikz}
\usetikzlibrary{arrows.meta,positioning,fit,calc,decorations.pathreplacing,shapes.multipart}
\usepackage{cleveref}

\theoremstyle{thmstyleone}
\newtheorem{theorem}{Theorem}[section]
\newtheorem{proposition}[theorem]{Proposition}

\newtheorem{corollary}[theorem]{Corollary}
\theoremstyle{definition}
\newtheorem{definition}[theorem]{Definition}
\newtheorem{assumption}[theorem]{Assumption}
\theoremstyle{remark}
\newtheorem{remark}[theorem]{Remark}

\begin{document}

\title{Chemotactic Feedback Controls Patterning in Hybrid Tumor–Stroma Model}

\author[1]{\fnm{Jiguang} \sur{Yu}}\email{jyu678@bu.edu}
\equalcont{These authors contributed equally to this work as co-first authors.}

\author*[2]{\fnm{Louis Shuo} \sur{Wang}}\email{swang116@vols.utk.edu}
\equalcont{These authors contributed equally to this work as co-first authors.}

\author[3]{\fnm{Zonghao} \sur{Liu}} \email{liuzonghao@fjmu.edu.cn}

\author*[3]{\fnm{Jingfeng} \sur{Liu}}\email{drjingfeng@126.com}

\affil[1]{\orgdiv{College of Engineering},
  \orgname{Boston University},
  \orgaddress{\city{Boston}, \postcode{02215}, \state{MA}, \country{United States}}}

\affil[2]{\orgdiv{Department of Mathematics},
  \orgname{University of Tennessee},
  \orgaddress{\city{Knoxville}, \postcode{37996}, \state{TN}, \country{United States}}}

\affil[3]{%
  \orgdiv{Department of Hepatopancreatobiliary Surgery},
  \orgname{Fujian Cancer Hospital},
  \orgaddress{\city{Fuzhou}, \postcode{350014}, \state{Fujian},\country{China}}}

\abstract{
Motivated by an ongoing collaboration with clinical oncologists and pathologists, we develop a hybrid partial differential equation--ordinary differential equation (PDE--ODE) framework that captures (i) competition between susceptible and resistant phenotypes, (ii) stromal state switching, and (iii) a clinically realistic open-loop, single-dose therapeutic agent $I$ with diffusion and clearance.

Clinical management of solid tumors is increasingly limited by spatial heterogeneity and therapy-induced resistance niches that are difficult to predict from well-mixed models.  We establish a rigorous mathematical backbone with forward invariance of the nonnegative cone and global-in-time well-posedness. Exploiting the decoupled drug equation $\partial_t I=d_I\Delta I-\gamma_I I$, we prove a long-time reduction during washout and show that the damped base dynamics admit no diffusion-driven (Turing-type) instability. We then formulate a directionality--damping principle: unidirectional (open-loop) sensing yields at most transient focusing, whereas bidirectional (closed-loop) feedback reshapes the effective mobility and produces explicit thresholds separating stable homogeneity, finite-band patterning (resistance niche formation), and aggregation when strong parabolicity is violated. Reproducible simulations corroborate this classification and highlight when flux regularization is required for physical realism.
}


\maketitle

\section{Introduction}

\subsection{Background and motivation}
\label{subsec:background}

\begin{figure}[t]
\centering
\resizebox{1.1\textwidth}{!}{%
\begin{tikzpicture}[
    font=\footnotesize,
    >=Latex,
    node distance=0.8cm and 0.5cm, % Vertical and Horizontal spacing
    % Style definitions
    box/.style={
        rounded corners=2mm,
        draw,
        very thick,
        align=center,
        inner sep=6pt,
        fill=white
    },
    subbox/.style={
        rounded corners=1.5mm,
        draw,
        thick,
        align=left,
        inner sep=6pt,
        fill=white,
        anchor=north west % Important for top-alignment
    },
    pill/.style={
        rounded corners=5mm,
        draw,
        thick,
        align=center,
        inner xsep=8pt,
        inner ysep=4pt,
        fill=white
    },
    arr/.style={->, thick},
    garr/.style={->, thick, dotted},
    container/.style={
        draw,
        very thick,
        rounded corners=4mm,
        inner sep=12pt
    }
]

\node[box, text width=5.8cm] (tumor) {
    \textbf{Tumor layer (PDE)}\\[3pt]
    $\partial_t S = d_S\Delta S + \lambda_S S(1-\tfrac{S+R}{K}) - \alpha S - \delta(I)S + \xi[1-\phi(I)]R$\\[3pt]
    $\partial_t R = d_R\Delta R + \lambda_R R(1-\tfrac{S+R}{K}) + \alpha S + \eta\,\phi(I)F_aR - \xi[1-\phi(I)]R$
};

\node[subbox, above=1cm of tumor, anchor=center] (switchSR) {
    $S \xrightleftharpoons[\xi(1-\phi(I))]{\alpha} R$
};
\draw[garr] (switchSR.south) -- (tumor.north);

\node[box, right=1.2cm of tumor, text width=3.2cm] (drug) {
    \textbf{Drug (PDE)}\\[3pt]
    $\partial_t I = d_I\Delta I - \gamma_I I$\\[3pt]
    $\delta(I)=\delta_0 \frac{I}{I+K_I}$\\[2pt]
    $\phi(I)\in(0,1)$\\[2pt]
    $\phi(0)\approx0,\ \phi(\infty)\approx1$
};

\node[box, right=0.8cm of drug, text width=4.8cm] (stroma) {
    \textbf{Stroma switch (ODE)}\\[3pt]
    $\partial_t P = -\theta\phi(I)P+\beta[1-\phi(I)]F_a$\\[2pt]
    $\partial_t F_a =\theta\phi(I)P-\beta[1-\phi(I)]F_a$\\[4pt]
    $\boxed{\partial_t(P+F_a) = 0 \Rightarrow P{+}F_a{=}P_T}$
};

\node[subbox, below=0.8cm of tumor, text width=5.8cm, align=center, anchor=north] (damp) {
    \textbf{Damping backbone}\\
    logistic crowding $\Rightarrow$ dissipative kinetics\\
    drug washout $\|I(t)\|_\infty\le e^{-\gamma_I t}\|I_0\|_\infty$
};

\node[pill, below=0.8cm of drug] (hub) {
    \textbf{Signaling hub}\\
    $I\Rightarrow(\delta,\phi)$
};

\draw[arr] (drug.west) -- node[above, font=\scriptsize] {$-\delta(I)S$} (tumor.east);
\draw[arr] (drug.south) -- (hub.north);
\draw[arr] (hub.east) -- node[above, sloped, font=\scriptsize] {$\phi(I)$} (stroma.south west);

\draw[arr] (stroma.north west) to[bend right=25] coordinate[midway] (curveApex) node[midway, above, font=\scriptsize] {$\eta\,\phi(I)F_aR$} (tumor.north east);

\draw[arr] (tumor.south) -- (damp.north);
\draw[arr] (drug.south west) to[out=240, in=60] (damp.north east);

\node[container, fit=(tumor)(drug)(stroma)(switchSR)(damp)(curveApex)] (topFrame) {};

\node[font=\bfseries, anchor=north, yshift=-6pt] at (topFrame.north) {Hybrid PDE--ODE coupling topology};

\node[below=1.5cm of topFrame.south west, anchor=west, font=\bfseries] (principleTitle) {Directionality--Damping Principle (mechanism map)};

\node[subbox, below=0.3cm of principleTitle, text width=5.0cm] (col1_top) {
    \textbf{Baseline (no chemotaxis)}\\
    $I(t)\to0 \Rightarrow \delta,\phi\to0$\\
    $\Rightarrow$ damped $(S,R)$ RD\\
    $\Rightarrow$ \textbf{no Neumann Turing}
};

\node[subbox, right=0.5cm of col1_top, text width=5.0cm] (col2_top) {
    \textbf{Unidirectional sensing (open-loop)}\\
    flux $-\nabla\!\cdot(\chi_W W\nabla c)$\\
    signal relaxes $\partial_c\mathcal Q<0$\\
    damping + relaxation $\Rightarrow$ stable
};

\node[subbox, right=0.5cm of col2_top, text width=5.2cm] (col3_top) {
    \textbf{Bidirectional feedback (closed-loop)}\\
    signal $g=g(S,R)$ creates gradients\\
    modifies transport operator\\
    $\mathcal A(\mu)=\mathcal J_F-\mu M$
};

\node[subbox, below=2.6cm of col1_top.north west, text width=5.0cm, anchor=north west] (col1_bot) {
    \textbf{Outcome: Homogenization}\\
    stable $\mathbb K^*$ \\spatial modes decay
};

\node[subbox, right=0.5cm of col1_bot, text width=5.0cm] (col2_bot) {
    \textbf{Outcome: Transient focusing}\\
    short-lived aggregation $\Rightarrow$ decay\\
    $\Rightarrow$ \textbf{cannot generate patterns}
};

\node[subbox, right=0.5cm of col2_bot, text width=5.2cm] (col3_bot) {
    \textbf{Three regimes}\\
    (I) $\mathrm{sym}(M)\succ0$ stable\\
    (II) finite unstable band $\Rightarrow$ \textbf{Turing}\\
    (III) $\mathrm{sym}(M)\not\succ0$ $\Rightarrow$ aggregation
};

\node[container, fit=(principleTitle)(col1_bot)(col3_bot)(col3_top)] (botFrame) {};

\end{tikzpicture}
}
\caption{\textbf{Unified mechanism diagram.} Top: hybrid PDE--ODE coupling topology for tumor phenotypes $(S,R)$, drug $I$, and stromal switching $(P,F_a)$ under Neumann boundary conditions. Bottom: the directionality--damping principle: (i) drug washout and logistic damping enforce baseline homogenization (no Turing); (ii) unidirectional open-loop sensing yields at most transient focusing; (iii) bidirectional feedback reshapes the effective mobility $M$ and yields stable, finite-band (Turing-type) niche formation, or aggregation/ill-posedness depending on strong parabolicity.}
\label{fig:unified_mechanism}
\end{figure}
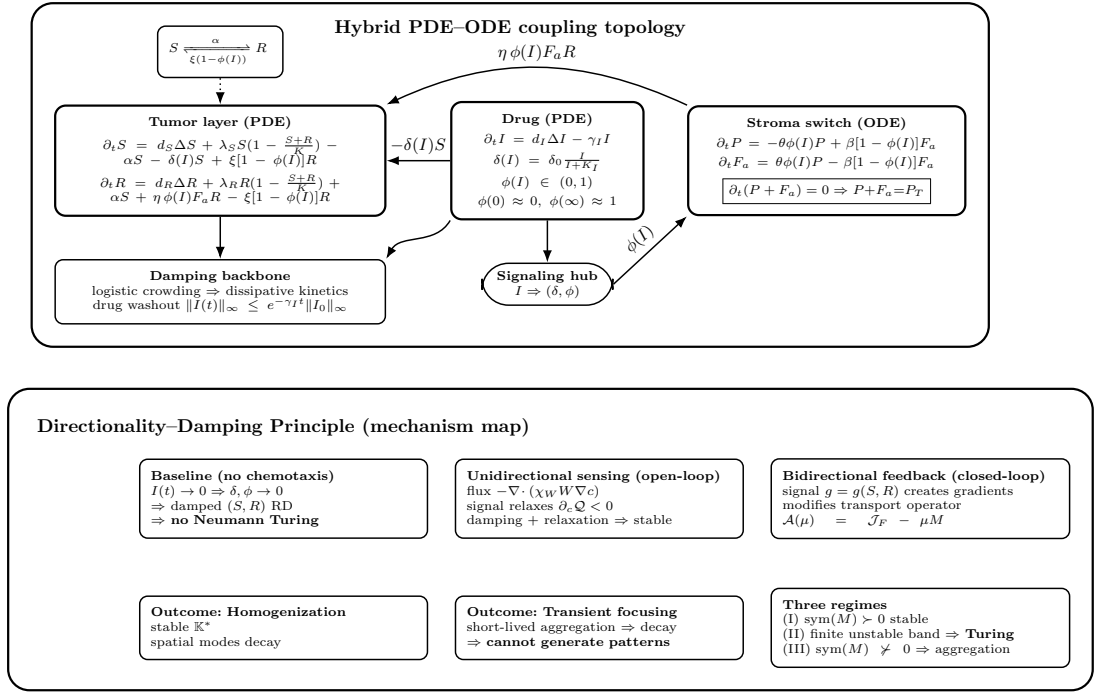

Self-organized spatial heterogeneity is observed across physical, chemical, and biological systems, where it often emerges from the interplay between local reaction kinetics, spatial transport, and environmental signaling \citep{cross1993pattern,dolnik2001resonant,murray2003mathematical,wang2026damage,colizza2007reaction,belik2011natural}. 
A central goal in the nonlinear analysis of partial differential equations is to identify which structural mechanisms can generate spatial patterns from homogeneous equilibria, and which mechanisms instead enforce relaxation toward homogeneity. 
In this paper we focus on a mechanism-level question motivated by tumor microenvironment dynamics: under strong damping (crowding/competition and signal relaxation), which coupling topologies can nevertheless amplify spatial perturbations into persistent heterogeneity, and which are intrinsically incapable of doing so?

The classical patterning paradigm is formulated as reaction--diffusion equations \citep{cherniha2017nonlinear,kondo2010reaction},
\begin{align*}
    \begin{cases}
    \partial_{t}X = d_{1}\Delta X + a(X-h)+b(Y-k), \\ 
    \partial_{t}Y = d_{2}\Delta Y + c(X-h)+d(Y-k),
    \end{cases}
\end{align*}
and, as initiated by Turing \citep{turing1952chemical}, shows that diffusion may destabilize a stable homogeneous steady state and induce a diffusion-driven bifurcation to patterned states. 
This mechanism has been extensively studied, and rigorous theories of existence, stability, and bifurcation for nonlinear diffusion systems are now classical \citep{chen2019non,wang2026breakdown,henry2006geometric,smoller2012shock}. 

Although classical Turing mechanisms generate spatial structures through diffusion-driven instabilities without requiring directed cell movement, pattern initiation in many biological systems is an active process in which cells sense, respond to, and amplify nascent spatial cues; the resulting directed migration then contributes directly to subsequent morphogenesis. 
Cell movement is therefore central to examples such as aggregation of the social amoeba \textit{Dictyostelium discoideum} \citep{hofer1995dictyostelium}, stripe formation in fish pigmentation \citep{painter1999stripe,wang2025multi}, gastrulation and limb morphogenesis in the chick embryo \citep{li1999cell,yu2026rigorous,yang2002cell}, various bacterial aggregation phenomena \citep{budrene1991complex,yu2026beyond,budrene1995dynamics}, and primitive streak formation \citep{painter2000chemotactic,gao2022rolling}. 
In these cases the movements are typically elicited and guided by chemotactic responses to spatially varying signals, a process termed chemotaxis. 
The modeling of chemotaxis begins with the Keller--Segel formulation \citep{keller1970initiation,keller1971model,yu2026microscopic,murray2007mathematical}, which incorporates directed motion along chemical gradients:
\begin{align}
\label{eq:Keller_Segel}
\begin{cases}
\partial_{t} n = f(n) - \nabla \cdot \left( n \chi(a) \nabla a \right) + \nabla \cdot D \nabla n, \\
\partial_{t} a = g(a,n) + \nabla \cdot D_{a} \nabla a.
\end{cases}
\end{align}
While diffusion tends to smooth perturbations, chemotaxis can destabilize homogeneous states in suitable regimes, leading to aggregation, spatial patterning, or blow-up phenomena \citep{horstmann2005boundedness,jin2025critical,cai2026optimal,horstmann20031970,arumugam2021keller,jager1992explosions}. 
Subsequent refinements enhance clinical relevance \citep{martiel1987model,goldbeter1996biochemical,wang2026elliptic,monk1989cyclic,spiro1997model}. 
Density-dependent motility and volume-filling effects have been modeled by a
   range of nonlinear diffusivities and saturating chemotactic sensitivities \citep{hofer1995dictyostelium,kowalczyk2005preventing,yu2026fullcovariance,choi2010prevention,lushnikov2008macroscopic,wang2010chemotaxis},
   which can produce either finite-wavelength patterns or aggregation/blow-up
   depending on the regime \citep{bubba2019chemotaxis,wang2010chemotaxis}.
Figure~\ref{fig:mechanism_comparison} illustrates this distinction: the left panel depicts diffusion-driven (Turing-type) patterning, whereas the right panel depicts an aggregation-driven instability with spike formation.

\begin{figure}
    \centering
    \includegraphics[width=0.75\linewidth]{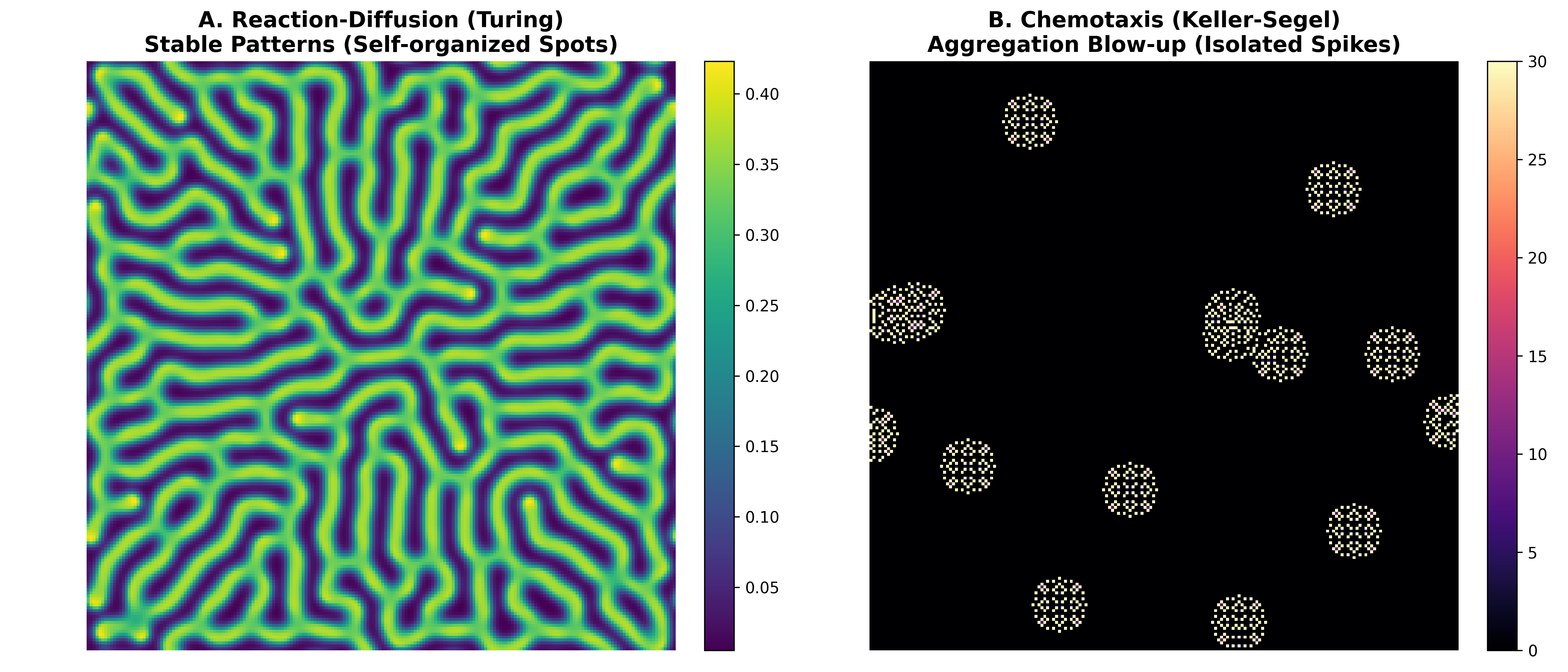}
    \caption{\textbf{Comparison of spatial instability mechanisms.} \textbf{(A)} Diffusion-driven instability (Turing mechanism): interaction between local reaction kinetics and differential diffusion generates stable, self-organized spatial patterns (e.g., stripes) from a homogeneous state. \textbf{(B)} Aggregation-driven instability (chemotaxis mechanism): directed transport leads to mass concentration and blow-up phenomena (spots) rather than smooth spatial structures. This paper investigates the transition between these regimes.}
    \label{fig:mechanism_comparison}
\end{figure}

In tumor microenvironment settings, spatial organization is shaped not only by diffusion and taxis, but also by heterogeneous signaling, competitive interactions, and state transitions of both tumor and stromal compartments \citep{farhood2018tumor,david2025role,zhang2024tumor,yuan2016spatial,wells2015spatial}. 
Chemotaxis-based and reaction--diffusion models have been used to study tumor growth and invasion \citep{bubba2019chemotaxis,yu2026age,watts2016pdgf,puliafito2015three}, yet many formulations treat microenvironmental components in a reduced or decoupled manner that obscures the mechanism by which spatial heterogeneity is created or eliminated \citep{oudin2017physical,yu2026pattern,esfahani2022three}. 
In particular, from a rigorous pattern-formation perspective, the joint role of (i) competing tumor phenotypes, (ii) microenvironmental state transitions, and (iii) fast-relaxing exogenous inhibitory forcing has not been fully characterized. 
Hybrid multiscale oncology models provide a complementary route to connect tissue-scale transport fields with cell-level phenotypic dynamics and resistance evolution \citep{wang2025analysis1}. 
These developments support the broader modeling strategy pursued here: identifying mechanism-level organizing principles in a tractable hybrid continuum framework, while clarifying how additional stochastic or agent-based structure could be incorporated in a mathematically controlled manner.

\subsection{Model and biological interpretation}
\label{subsec:model_bio}

The above research gap motivates the present work. 
We develop a mathematically tractable hybrid PDE--ODE framework that couples competition between two tumor phenotypes to stromal state transitions, and includes an open-loop, single-dose therapeutic agent $I$ that diffuses in space and is cleared in time. 
\Cref{fig:unified_mechanism} illustrates the coupling topology considered here. 
Our goal is mechanism-driven: under strong damping (crowding/competition and signal relaxation), when can directed transport generate spatial heterogeneity from a homogeneous equilibrium, and when is it structurally incapable of inducing diffusion-driven (Turing-type) bifurcation?

A closely related issue is physical realism. 
When feedback is sufficiently strong to overwhelm diffusive smoothing, unregularized continuum descriptions may predict aggregation/collapse rather than smooth, finite-wavelength structures, indicating the need for biologically motivated regularization (e.g.\ volume filling \citep{wrzosek2010volume,liu2026ftu} or flux saturation \citep{chertock2012chemotaxis,liu2026computational}). 
These themes guide both the modeling hierarchy and the analytical results developed in the remainder of the paper.

To study these questions in a form that supports rigorous analysis and transparent interpretation, we consider a hybrid reaction--diffusion PDE--ODE model on a bounded domain $U\subseteq\mathbb{R}^N$ ($N\in\{1,2,3\}$) with $C^\infty$ boundary. 
Throughout the paper, we impose homogeneous Neumann boundary conditions (zero flux) for the diffusing species,
\[
\partial_{\mathrm n}S=\partial_{\mathrm n}R=\partial_{\mathrm n}I=0
\qquad \text{on }\partial U,
\]
which represent a closed tissue region with no net exchange across the boundary. 
We assume nonnegative initial data, consistent with the interpretation of the state variables as densities and concentrations.

\medskip
\noindent\textbf{Base hybrid PDE--ODE system (no chemotaxis).}
We propose the following coupled system:
\begin{align}
\label{eq:system}
\begin{cases}
\partial_{t} S = d_{S}\Delta S + \lambda_{S} S \left(1 - \tfrac{S+R}{K}\right) - \alpha S - \delta(I)S + \xi \bigl[1 - \phi(I)\bigr]R, & x\in U,\ t>0, \\[4pt]
\partial_{t} R = d_{R}\Delta R + \lambda_{R} R \left(1 - \tfrac{S+R}{K}\right) + \alpha S + \eta \phi(I)F_aR - \xi \bigl[1 - \phi(I)\bigr]R, & x\in U,\ t>0, \\[4pt]
\partial_{t} I = d_{I}\Delta I - \gamma_{I}I, & x\in U,\ t>0, \\[4pt]
\partial_{t} P = -\theta \phi(I)P + \beta \bigl[1 - \phi(I)\bigr]F_a, & x \in U,\ t>0, \\[4pt]
\partial_{t} F_a = \theta \phi(I)P - \beta \bigl[1 - \phi(I)\bigr]F_a, & x \in U,\ t>0.
\end{cases}
\end{align}
The densities $S(x,t)$ and $R(x,t)$ represent two competing tumor phenotypes, interpreted as therapy-susceptible and therapy-resistant populations, respectively. 
Both $S$ and $R$ undergo logistic growth with intrinsic rates $\lambda_S,\lambda_R$ and a shared carrying capacity $K$, introducing strong damping through crowding and competition via the factor $1-(S+R)/K$.
The two phenotypes interconvert at rates $\alpha$ ($S\to R$) and $\xi$ ($R\to S$), representing reversible switching under environmental pressure.

\medskip
\noindent\textbf{Stromal switching variables.}
To ground the microenvironmental variables in a concrete setting, we interpret $(P,F_a)$ as two phenotypic states of the stromal fibroblast pool. 
Here $P(x,t)$ denotes the density of quiescent (inactive) fibroblasts, while $F_a(x,t)$ denotes the density of activated fibroblasts (e.g.\ cancer-associated fibroblasts \citep{sahai2020framework,yu2026size}). 
Activated fibroblasts modulate the local microenvironment via remodeling cues and growth-factor secretion. 
With this identification, the model couples three mechanisms central to tumor microenvironments: competition, signal-mediated regulation, and microenvironmental state transitions \citep{farhood2018tumor,david2025role,zhang2024tumor,liang2026separation,yuan2016spatial,wells2015spatial}. 
The $(P,F_a)$ subsystem has the structure of a two-state phenotypic conversion model with drug-dependent rates,
\[
P \xrightleftharpoons[\ \beta(1-\phi(I))\ ]{\ \theta\phi(I)\ } F_a.
\]
In the hybrid regime, we treat the stromal conversion variables $(P,F_a)$ as effectively nonmotile on the pattern-formation time window and set
\begin{equation*}
    \partial_t P = -\theta\phi(I)P + \beta [1-\phi(I)]F_a,\qquad \partial_t F_a = \theta\phi(I)P-\beta [1-\phi(I)]F_a,\qquad x\in U, \ t>0,
\end{equation*}
so that the $P$- and $F_a$-equations become pointwise ODEs coupled to the PDE components $(S,R,I)$. 
The conversion structures imply the pointwise conservation law $\partial_t(P+F_a)=0$ (see Proposition~\ref{prop:IPA_bounds}). 

\subsubsection{Timescale separation and the non-motile stromal limit}
\label{subsubsec:timescales}
The pointwise (non-motile, conservative) stromal module is the leading order of
an explicit slow--fast reduction, which we now make precise; this also fixes the
magnitudes under which the following source--sink terms \eqref{eq:source_sink} may be
separated from the fast switching:
\begin{align}
\label{eq:source_sink}
\partial_t P = \dots + s_P - \mu_P P,
\qquad
\partial_t F_a = \dots + s_{F_a} - \mu_{F_a} F_a,
\end{align}

The fast timescale is set by the relaxation rate of the switching subsystem.
Writing the switching module as $\dot{(P,F_a)}^{\mathsf T}=B(I)(P,F_a)^{\mathsf T}$ with
\[
B(I)=\begin{pmatrix}-\theta\phi & \beta(1-\phi)\\ \theta\phi & -\beta(1-\phi)\end{pmatrix},
\qquad
\operatorname{spec}B(I)=\{0,\ -k_{\mathrm{sw}}\},\quad
k_{\mathrm{sw}}:=\theta\phi+\beta(1-\phi),
\]
the zero eigenvalue is the conserved pool $P+F_a$ and the nonzero eigenvalue
$-k_{\mathrm{sw}}$ is the rate at which the pool collapses onto its
quasi-steady manifold. Hence the fast switching timescale is
$\tau_{\mathrm{sw}}:=k_{\mathrm{sw}}^{-1}=\bigl(\theta\phi+\beta(1-\phi)\bigr)^{-1}$;
note that $k_{\mathrm{sw}}$ interpolates linearly between $\beta$ (at $\phi=0$) and
$\theta$ (at $\phi=1$), so it is uniformly bounded below,
$k_{\mathrm{sw}}\ge\min\{\theta,\beta\}>0$ for all $\phi\in[0,1]$; in particular it
does not degenerate in the washout limit $\phi\to0$, where $k_{\mathrm{sw}}\to\beta$.

Three processes must be distinguished. Drug-driven fibroblast phenotype
switching $P\rightleftharpoons F_a$ relaxes on this fast timescale
$\tau_{\mathrm{sw}}$ (minutes--hours). Tumor migration, growth, and the chemotactic
instability develop on the pattern-formation window $\tau_{\mathrm{pat}}$
(hours--days). Fibroblast spatial redistribution by migration, and
fibroblast proliferation and turnover, are slow relative to $\tau_{\mathrm{pat}}$;
the relevant measure is the diffusive spread of fibroblasts over the instability
window, $\ell_F=\sqrt{d_F\,\tau_{\mathrm{pat}}}$, which for reported fibroblast
motilities is small compared with the selected pattern wavelength. Thus the
separation is not ``tumor slow, stroma fast'' wholesale; rather, stromal switching is fast while stromal transport and turnover are slow, both relative to the pattern-formation window.

Introduce the small parameter
\[
\varepsilon:=\frac{\tau_{\mathrm{sw}}}{\tau_{\mathrm{pat}}}\ll 1,
\]
and nondimensionalize time on the switching scale. Writing the stromal equations
with motility and the source--sink corrections \eqref{eq:source_sink},
\[
\partial_t P = \underbrace{-\theta\phi(I)P+\beta[1-\phi(I)]F_a}_{O(1)}
+\underbrace{d_P\Delta P + s_P-\mu_P P}_{O(\varepsilon)},
\]
and symmetrically for $F_a$, the transport and turnover terms are $O(\varepsilon)$
relative to the switching relaxation rate $k_{\mathrm{sw}}$ provided
\begin{equation}
\label{eq:timescale_ordering}
\frac{d_P}{L^2},\ \frac{d_{F_a}}{L^2},\ \mu_P,\ \mu_{F_a},\
\frac{s_P}{P_T},\ \frac{s_{F_a}}{P_T}
\ =\ O(\varepsilon)\cdot k_{\mathrm{sw}}
\ \ll\ k_{\mathrm{sw}}=\theta\phi+\beta(1-\phi),
\end{equation}
where $L$ is the domain scale. A convenient $\phi$-uniform (drug-independent)
sufficient condition, valid along the entire washout trajectory, replaces
$k_{\mathrm{sw}}$ by its floor $\min\{\theta,\beta\}$. At leading order in
$\varepsilon$ the fast switching equilibrates the conserved pool, giving the
pointwise module of \eqref{eq:system} with $P+F_a=P_T$
(Assumption~\ref{ass:homogeneous_pool}); the slow terms enter only as an
$O(\varepsilon)$ drift of $P_T$.

\begin{remark}[\textbf{Robustness of the classification under the slow terms}]
\label{rmk:slow_robustness}
Reinstating the $O(\varepsilon)$ terms in \eqref{eq:timescale_ordering} shifts the
homogeneous equilibrium and the linearization by $O(\varepsilon)$. Because the
Regime~I/II/III classification of
Sections~\ref{subsec:turing_stability}--\ref{subsec:twoway_classification}
depends only on the sign structure of the reaction Jacobian and on strong
parabolicity of the effective mobility $M$---both open conditions, stable under
$O(\varepsilon)$ perturbations---the qualitative classification is unchanged for
$\varepsilon$ sufficiently small. Fully motile fibroblast dynamics
($d_F=O(1)$), which can drive invasion, lie outside this reduction and are a
natural extension.
\end{remark}

\medskip
\noindent\textbf{Drug pharmacokinetics and signaling.}
The field $I(x,t)$ denotes the concentration of an exogenously administered therapeutic agent (single-dose chemotherapy drug). 
It diffuses through tissue and is cleared at rate $\gamma_I$, so that after dosing the drug relaxes exponentially. 
Repeated dosing or sustained delivery can be incorporated by adding an input term, which we do not pursue here.
The drug $I$ enters the tumor layer through an inhibitory term acting on $S$, modeled by Michaelis--Menten kinetics \citep{cornish2015one},
\[
\delta(I)=\delta_0\,\frac{I}{I+K_I},
\]
with maximal inhibition rate $\delta_0$ and half-saturation concentration $K_I$.
The same drug field controls stromal activation through a switch $\phi$.
We assume $\phi: \mathbb{R}_+\to (0,1)$ is smooth, bounded, and globally Lipschitz continuous. 
Additionally, we require $\phi(I)\approx 0$ for small $I$ and $\phi(I)\approx 1$ for large $I$. 
A representative choice is
\[
\phi(I)=\tanh(\kappa_{\phi} I),
\qquad \kappa_{\phi}>0,
\]
where $\kappa_{\phi}$ controls the steepness of the drug-dependent switch. 
Activated fibroblasts promote the resistant phenotype through the coupling term $\eta\,\phi(I)\,F_a\,R$, encoding the idea that drug-driven stromal activation can transiently create a niche that favors resistance \citep{milne2025role,meads2009environment}.
\Cref{fig:model_topology} summarizes the resulting interaction topology.

\begin{figure}
    \centering
    \includegraphics[width=0.75\linewidth]{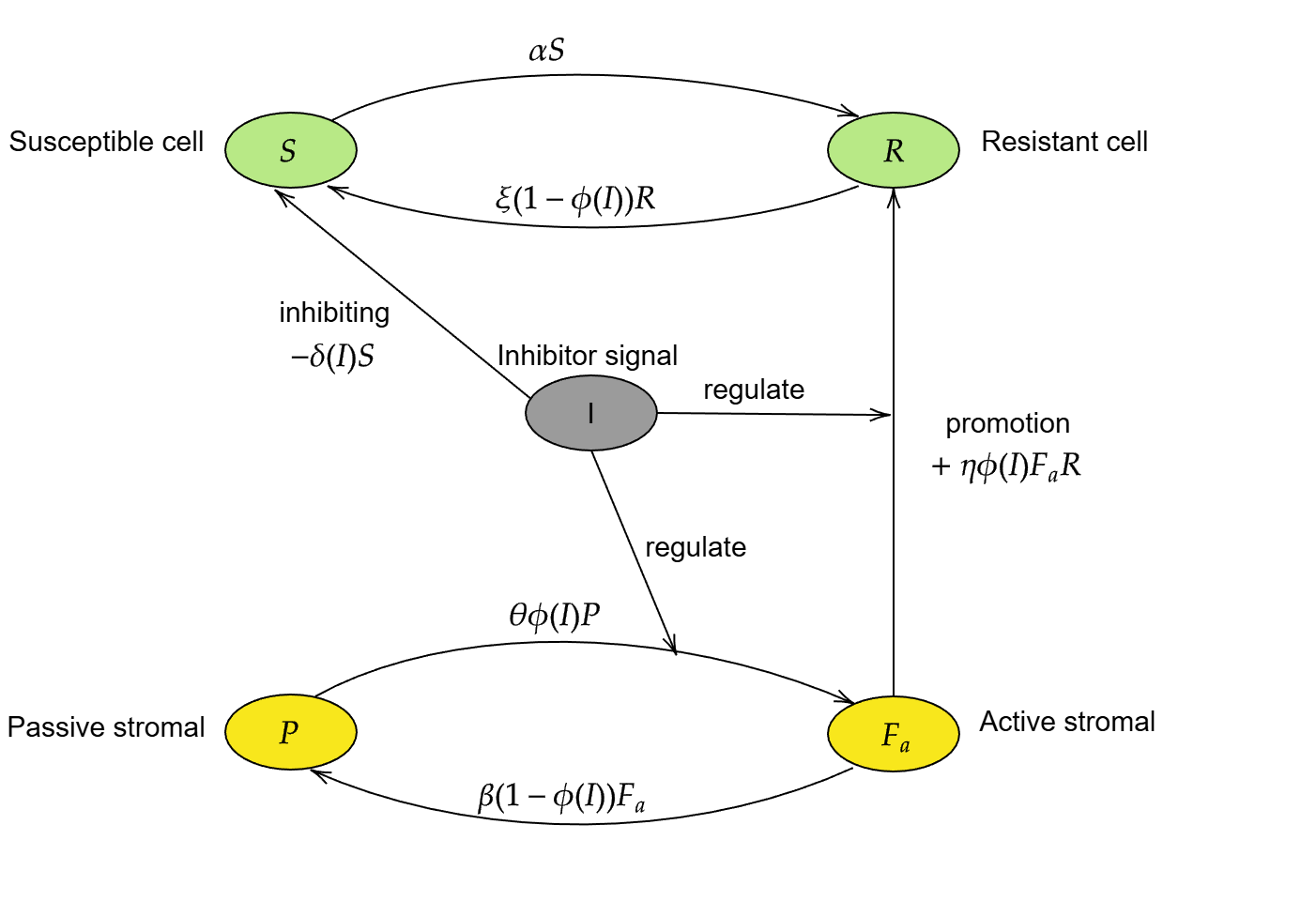}
    \caption{\textbf{Topology of the hybrid PDE--ODE interaction network.} The inhibitor $I$ acts as a central signaling hub connecting the tumor layer $(S,R)$ and the microenvironmental layer $(P,F_a)$. Key features include $I$ providing negative feedback to $S$ ($-\delta(I)S$) while activating the microenvironmental stromal cells via the switch $\phi(I)$. The switch promotes resistance via the $\eta\phi(I)F_aR$ term, creating a drug-gated feedback channel from stroma to the resistant phenotype.}
    \label{fig:model_topology}
\end{figure}

\medskip
\noindent\textbf{Two modeling levels: baseline vs.\ chemotaxis extensions.}
We analyze two levels of description:
\begin{enumerate}[label=(\roman*)]
\item the base hybrid PDE--ODE system \eqref{eq:system} (no chemotaxis), for which $P,F_a$ is governed by pointwise ODEs, and yielding a conservative stromal pool; and
\item chemotaxis-augmented extensions (Section~\ref{sec:chemotaxis}), in which spatial gradients of a signal appear explicitly in the flux and therefore require sufficient spatial regularity of that signal. 
We denote the chemotactic signal by $c$ and evolve it by a parabolic equation in the chemotaxis setting (e.g.\ $\partial_t c=d_c\Delta c+\mathcal{Q}(c)$). 
\end{enumerate}

\medskip
\noindent\textbf{Homogeneous pool branch.}
Governed by pointwise ODEs, the conserved quantity $P(x,t)+F_a(x,t)=P_0(x)+F_{a,0}(x)$ may be spatially heterogeneous, and the system admits a continuum of local steady states parameterized by this pool. 
In the stability and pattern-formation analysis below, we therefore restrict attention to the homogeneous branch corresponding to spatially constant total stromal density (see also the definition of homogeneous equilibrium in Section~\ref{subsec:turing_stability}).

\subsection{Main message and contributions: Bridging theory and a clinical constraint}
\label{subsec:main_contrib}

This work is motivated by a recurring translational tension:
clinically relevant interventions often act as transient, open-loop forcings, whereas spatial resistance niches are often sustained by endogenous, feedback-mediated microenvironmental remodeling. 
Our central message is mechanism-level: in damped hybrid PDE--ODE systems, whether therapeutic perturbations relax toward homogeneity or instead permit spatial self-organization is determined by the interplay between coupling directionality (open-loop vs.\ closed-loop) and damping/relaxation structure (crowding competition, signal decay, and washout).

Unlike standard models that assume steady or sustained forcing, our framework incorporates the clinically relevant constraint of single-dose pharmacokinetics (one of several clinical constraints; others, such as motile stroma, immune interactions, quantitative pharmacokinetics/pharmacodynamics and patient-specific parameters, are outside the present scope, see \Cref{subsec:clinical_interpretation}) through a decaying inhibitor field $I$. 
In this stringent baseline, we show that directed transport can overcome relaxation only when it participates in a feedback loop that reshapes the effective transport operator---a structural insight that suggests testable hypotheses about when stromal crosstalk may support, or fail to support, persistent spatial heterogeneity.

\medskip
\noindent\textbf{Main results.}
For a bounded domain $U\subset\mathbb{R}^N$ with smooth boundary and homogeneous Neumann conditions, our analysis yields three contributions:

\begin{enumerate}[label=(R\arabic*)]
\item \textbf{Rigorous mathematical backbone for biological consistency.}
For the base hybrid system \eqref{eq:system}, we establish global well-posedness and forward invariance of the biologically meaningful state space. 
Using quasi-positivity, we prove preservation of nonnegativity for all components. 
Global-in-time classical solutions are obtained via an $L^\infty$ continuation (blow-up alternative) closed by explicit a priori bounds and comparison estimates. 
This ensures that any instability observed in later chemotaxis extensions reflects genuine dynamical behavior rather than finite-time mathematical pathology.

\item \textbf{Therapeutic washout limit and the baseline no-Turing conclusion.}
We analyze the long-time reduction under realistic drug clearance,
\[
\partial_t I=d_I\Delta I-\gamma_I I,
\]
which implies exponential washout of the open-loop inhibitor. 
As $I(\cdot,t)\to 0$, all drug-mediated couplings vanish asymptotically and the dynamics reduce to the $I\equiv 0$ limit. 
In particular, the post-washout reaction--diffusion linearization of the tumor subsystem $(S,R)$ about its homogeneous coexistence equilibrium admits no diffusion-driven Turing instability. 
Thus, in the damped baseline architecture, diffusion cannot destabilize homogeneity after the transient drug perturbation has relaxed. 
This establishes a stringent reference point for identifying which additional transport couplings can genuinely generate spatial structure.

\item \textbf{Directionality--damping principle for directed transport.}
To model spatial reorganization, we introduce chemotaxis extensions and establish a dichotomy governed by coupling topology (Section~\ref{subsec:mechanism_statement}). 
Under damping and signal relaxation, unidirectional sensing cannot generate patterns from homogeneity, whereas closed-loop feedback can, with strong parabolicity of $M$ separating finite-band selection from aggregation.
This classification provides a mathematically explicit mechanism map linking feedback strength to the qualitative form of spatial instability.
\end{enumerate}

\medskip
\noindent\textbf{Numerical validation.}
Reproducible simulations corroborate the analytical regimes: they confirm relaxation toward homogeneity in the base and unidirectional settings, demonstrate sustained finite-wavelength structure only under closed-loop feedback in the well-posed regime, and identify the regularization (e.g.\ flux limitation) required to obtain physically meaningful dynamics in the ill-posed regime.

\medskip
\noindent\textbf{Relation to previous work and novelty.}
The individual analytical tools we use---Routh--Hurwitz mode analysis, LaSalle's
invariance principle, analytic-semigroup well-posedness, and Neumann-mode
dispersion relations---are classical \citep{murray2003mathematical}. Our
contribution is not new PDE technology but an organizing principle obtained by
combining three elements not, to our knowledge, previously assembled for hybrid
tumor--stroma systems. First, a single-dose washout reduction
(Section~\ref{subsec:long_time_reduction}) converts an open-loop, nonautonomous
system into a rigorous autonomous $I\equiv0$ limit with a proven no-Turing
baseline, supplying a well-characterized reference state rather than an assumed
one. Second, a directionality--damping dichotomy stated at theorem level:
under damping, open-loop sensing provably cannot generate patterns from
homogeneity (Theorem~\ref{thm:oneway}, Corollaries~\ref{cor:oneway_logis_damping}--\ref{cor:oneway_damping}), whereas closed-loop feedback can
(\Cref{thm:twoway_classification}). Prior chemotaxis and PDE--ODE cancer models
\citep{wang2010chemotaxis,bubba2019chemotaxis,fernandez2022glioblastoma,stinner2014global}
typically fix a signal source or study a single coupling, and do not classify how
coupling directionality interacts with damping to decide pattern
generation. Third, the effective mobility matrix $M=\mathcal D_{\mathrm{diff}}-\mathcal H(\mathbf u^*)$ and its strong
parabolicity provide a single linear-algebraic criterion separating finite-band
(Turing) selection from Hadamard/aggregation ill-posedness---a distinction often
conflated in chemotactic pattern claims. The value of the paper is this
classification map from feedback structure to the form of spatial instability,
not the difficulty of any individual proof.

\medskip
\noindent\textbf{Organization of the paper.}
\Cref{sec:models} introduces the base hybrid system, the analytical setting, and the chemotaxis model classes. 
\Cref{sec:base_analysis} establishes invariance, global well-posedness, the washout reduction, and the baseline no-Turing conclusion for the post-washout $(S,R)$ reaction--diffusion linearization. 
\Cref{sec:chemotaxis} develops the directionality--damping principle: unidirectional stability and the bidirectional classification into stable, finite-band (Turing-type), and aggregation/ill-posed regimes. 
\Cref{sec:numerical} provides reproducible numerical evidence supporting the classification and illustrates a minimal feedback experiment that restores finite-wavelength patterning. 
We conclude in the Discussion with interpretation, limitations, and modeling extensions.

\section{Model and Analytical Setting}\label{sec:models}
\subsection{Base hybrid PDE--ODE model (no chemotaxis)}
\label{subsec:base_model}

We begin with the base hybrid PDE--ODE system \eqref{eq:system} without chemotactic transport. 
This level of description isolates the competition--inhibition--conversion architecture and provides the analytical baseline for the mechanism results developed later. 
In particular, it allows us to establish invariance of the nonnegative cone, global well-posedness, long-time reduction under drug washout, and a baseline no-Turing conclusion for the damped reaction--diffusion core.

Let $U\subseteq\mathbb{R}^N$ ($N\in\{1,2,3\}$) be a bounded domain with $C^\infty$ boundary $\partial U$. 
Throughout the paper, we impose homogeneous Neumann boundary conditions for the diffusing components $(S,R,I)$,
\begin{equation}
\label{eq:neumann_bc}
\partial_{\mathrm n}S=\partial_{\mathrm n}R=\partial_{\mathrm n}I=0
\qquad\text{on }\partial U\times(0,\infty),
\end{equation}
corresponding to a closed tissue region with zero flux across the boundary.

\medskip
\noindent\textbf{Hybrid structure and baseline regime.}
The model \eqref{eq:system} consists of reaction--diffusion equations for $(S,R,I)$ coupled to a stromal conversion subsystem $(P,F_a)$, together with \eqref{eq:neumann_bc} and nonnegative initial data. 
In the hybrid regime, $(P,F_a)$ evolves pointwise in space as an ODE subsystem driven by the drug-dependent switch $\phi(I)$. 
This conservative switching limit retains microenvironmental state transitions; small source--sink corrections can be incorporated on longer timescales as in \eqref{eq:source_sink}.

Let $\mathbf u=(S,R,I,P,F_a)^{\mathsf T}$ and write \eqref{eq:system} in the semilinear form
\begin{align}
\label{eq:base_vector_form}
\partial_t\mathbf u-\mathcal D_{\mathrm{diff}}\Delta \mathbf u=\mathcal G(\mathbf u),
\qquad 
\mathcal D_{\mathrm{diff}}=\mathrm{diag}(d_S,d_R,d_I,0,0),
\end{align}
where $\mathcal G=(G_1,\dots,G_5)^{\mathsf T}$ is given by
\begin{align*}
\begin{cases}
G_1(\mathbf{u}) = \lambda_S S\left(1-\tfrac{S+R}{K}\right)-\alpha S-\delta(I)S+\xi(1-\phi(I))R,\\[4pt]
G_2(\mathbf{u}) = \lambda_R R\left(1-\tfrac{S+R}{K}\right)+\alpha S+\eta\,\phi(I)F_aR-\xi(1-\phi(I))R,\\[4pt]
G_3(\mathbf{u}) = -\gamma_I I,\\[4pt]
G_4(\mathbf{u}) = -\theta \phi(I)P+\beta(1-\phi(I))F_a,\\[4pt]
G_5(\mathbf{u}) = \theta \phi(I)P-\beta(1-\phi(I))F_a.
\end{cases}
\end{align*}
Here $\delta(I)=\delta_0 I/(I+K_I)$ is of Michaelis--Menten type \citep{cornish2015one} and $\phi(I)$ is a smooth switch (see Section~\ref{subsec:model_bio}). 
On bounded subsets of $\mathbb{R}_{\ge 0}$ these nonlinearities are smooth, so the reaction map $\mathcal G$ is locally Lipschitz. 
Table~\ref{tab:parameter_units} summarizes key parameters.

\begin{table}[htbp]
\centering
\caption{Key parameters, meanings, and units (final notation). Parameters marked by $\dagger$ appear only in chemotaxis / feedback extensions.}
\label{tab:parameter_units}
\begin{tabular}{lll}
\toprule
Parameter & Meaning & Unit \\
\midrule
\multicolumn{3}{l}{\textbf{Diffusion / transport}}\\
$d_S,d_R,d_I$ & diffusion coefficients of $S,R,I$ & mm$^2$/s \\
$d_c^\dagger$ & diffusion coefficient of chemoattractant $c$ & mm$^2$/s \\
$\chi_S^\dagger,\chi_R^\dagger$ & chemotactic sensitivity of $S,R$ to $\nabla c$ (unidirectional/abstract templates) & mm$^2$/(s$\cdot$M) \\
$\chi_S^{\prime\,\dagger}$ & additional sensitivity of $S$ to $\nabla c$ in the explicit feedback model & mm$^2$/(s$\cdot$M) \\
$\alpha_c^\dagger$ & flux-saturation strength in $\nabla c/(1+\alpha_c|\nabla c|)$ & mm/M \\
\midrule
\multicolumn{3}{l}{\textbf{Tumor kinetics (logistic competition and switching)}}\\
$\lambda_S,\lambda_R$ & intrinsic growth rates of $S,R$ & 1/s \\
$K$ & shared carrying capacity for $S+R$ & $M$ \\
$\alpha$ & phenotype switching rate $S\to R$ & 1/s \\
$\xi$ & phenotype switching rate $R\to S$ (drug-off switch) & 1/s \\
$\eta$ & promotion rate of $R$ by activated stroma (coupling via $F_a$) & 1/(s$\cdot$M) \\
\midrule
\multicolumn{3}{l}{\textbf{Drug dynamics and drug-mediated effects}}\\
$\gamma_I$ & drug clearance/decay rate in $\partial_t I=d_I\Delta I-\gamma_I I$ & 1/s \\
$\delta_0$ & maximal inhibition rate of $S$ by drug (via $\delta(I)$) & 1/s \\
$K_I$ & half-saturation concentration in $\delta(I)=\delta_0 I/(I+K_I)$ & $M$ \\
$\kappa_\phi$ & steepness parameter in $\phi(I)=\tanh(\kappa_\phi I)$ (or analogous switch) & 1/M \\
\midrule
\multicolumn{3}{l}{\textbf{Stromal switching module (baseline hybrid regime)}}\\
$\theta$ & activation rate $P\to F_a$ modulated by $\phi(I)$ & 1/s \\
$\beta$ & deactivation rate $F_a\to P$ modulated by $1-\phi(I)$ & 1/s \\
\midrule
\multicolumn{3}{l}{\textbf{Chemoattractant dynamics (feedback experiments)}}\\
$\kappa_c^\dagger$ & production rate of chemoattractant $c$ (e.g.\ $\kappa_c S$) & 1/s \\
$\rho_c^\dagger$ & decay rate of chemoattractant $c$ & 1/s \\
\bottomrule
\end{tabular}
\end{table}

\begin{assumption}[\textbf{Homogeneous stromal pool}]
\label{ass:homogeneous_pool}
In the hybrid regime, the conversion $(P,F_a)$ subsystem implies the pointwise conservation law $\partial_t(P+F_a)=0$ and hence
\[
P(x,t)+F_a(x,t)=P_0(x)+F_{a,0}(x)
\qquad \text{for all } t\ge 0.
\]
When studying spatially homogeneous equilibria and their stability, we therefore restrict attention to initial data with spatially constant total stromal pool:
\begin{equation}
\label{eq:pool_constant}
P_0(x)+F_{a,0}(x)\equiv P_T\qquad \text{for some constant }P_T\ge 0.
\end{equation}
This assumption ensures the existence of a homogeneous equilibrium branch in the baseline hybrid regime and will be imposed whenever homogeneous stability and pattern-formation criteria are invoked.
\end{assumption}

\subsection{Weak formulation and function spaces}
\label{subsec:weak_formulation}

Fix $T>0$ and denote $U_T:=U\times(0,T]$. 
Following standard parabolic theory \citep{evans2022partial,amann1995linear}, we formulate weak solutions for the diffusing variables $(S,R,I)$ under homogeneous Neumann boundary conditions. 
Let
\begin{align*}
  V_{\mathfrak N}(T)
  := L^2(0,T;H^1(U))\cap H^1\!\left(0,T;(H^1(U))'\right),
\end{align*}
and seek
\[
S,R,I\in V_{\mathfrak N}(T).
\]

In the hybrid regime, the stromal conversion variables $(P,F_a)$ evolve pointwise in space as an ODE subsystem driven by $\phi(I)$. 
Accordingly, we assume
\[
P,F_a\in C^1\bigl([0,T];L^2(U)\bigr),
\]
which is compatible with the ODE structure and with the pointwise conservation law for $P+F_a$ (see Proposition~\ref{prop:IPA_bounds}).

\medskip
\noindent\textbf{Weak formulation.}
Let $\mathbf u=(S,R,I,P,F_a)^{\mathsf T}$ and let $\mathcal G(\mathbf u)=(G_1,\dots,G_5)^{\mathsf T}$ be the reaction field in \eqref{eq:base_vector_form}. 
A quintuple $(S,R,I,P,F_a)$ is a weak solution on $(0,T)$ if $S,R,I$ and $P,F_a$ satisfy the above regularity and, for every test function
\[
\varphi=(\varphi_S,\varphi_R,\varphi_I)\in [V_{\mathfrak N}(T)]^3
\]
and for a.e.\ $t\in(0,T)$, the following identities hold:
\begin{align}
\label{eq:weak}
\begin{aligned}
\langle \partial_t S,\varphi_S\rangle + d_S(\nabla S,\nabla\varphi_S)_{L^2(U)} &= (G_1(\mathbf{u}),\varphi_S)_{L^2(U)},\\
\langle \partial_t R,\varphi_R\rangle + d_R(\nabla R,\nabla\varphi_R)_{L^2(U)} &= (G_2(\mathbf{u}),\varphi_R)_{L^2(U)},\\
\langle \partial_t I,\varphi_I\rangle + d_I(\nabla I,\nabla\varphi_I)_{L^2(U)} &= (G_3(\mathbf{u}),\varphi_I)_{L^2(U)}.
\end{aligned}
\end{align}
Here $(\cdot,\cdot)_{L^2(U)}$ denotes the $L^2(U)$ inner product and $\langle\cdot,\cdot\rangle$ denotes the duality pairing between $H^1(U)$ and $(H^1(U))'$. 
The homogeneous Neumann boundary conditions $\partial_{\mathrm n}S=\partial_{\mathrm n}R=\partial_{\mathrm n}I=0$ are encoded in the integration-by-parts step leading to \eqref{eq:weak}.

\medskip
\noindent\textbf{Initial data.}
We assume nonnegative initial data
\[
\mathbf u_0=(S_0,R_0,I_0,P_0,F_{a,0})\in (L^\infty(U))^5,
\qquad \mathbf u_0\ge 0\ \text{a.e.\ in }U.
\]

\subsection{Chemotaxis extensions and well-definedness assumptions}
\label{subsec:chemotaxis_models}

Starting from the base system \eqref{eq:system}, we introduce Keller--Segel-type fluxes for $S$ and $R$ along gradients of a signal field $c$. 
We use $c$ exclusively for chemotactic signaling. 
Since the fluxes depend on $\nabla c$, well-definedness requires sufficient spatial regularity of $c$, which we obtain by evolving $c$ via a parabolic relaxation equation.

We note that the components $(I,P,F_a)$ of \eqref{eq:system} are retained in the
concrete extension \eqref{eq:oneway_damping}; they are omitted from the abstract
templates \eqref{eq:oneway_logis_damping}, \eqref{eq:oneway} and \eqref{eq:twoway} only because they are spectrally inert at a homogeneous equilibrium---the mode-wise
linearization is block-triangular, with the drug and signal blocks strictly
stable and the conserved non-motile stromal block contributing $\{0,
-(\theta\phi^*+\beta(1-\phi^*))\}$ (the zero being the neutral conserved-pool
mode). Consequently mode-wise stability is governed by the $(S,R)$/signal block;
the full computation is given in Appendix~\ref{app:proof_corollary_oneway_damping}.

\medskip
\noindent\textbf{Unidirectional (open-loop) sensing.}
We first consider a one-way coupling in which tumor populations respond to the signal $c$ but do not regulate it. 
In this open-loop setting, the signal dynamics are independent of $(S,R)$ and relax in time. 
A concrete tumor--drug--stroma unidirectional extension takes the form
\begin{align}
\label{eq:oneway_damping}
\begin{cases}
\partial_t S = d_S\Delta S - \chi_S\nabla\cdot(S\nabla c) 
+ \lambda_S S\left(1 - \tfrac{S+R}{K}\right) - \alpha S - \delta(I)S + \xi(1 - \phi(I))R, \\[4pt]
\partial_t R = d_R\Delta R - \chi_R\nabla\cdot(R\nabla c) 
+ \lambda_R R\left(1 - \tfrac{S+R}{K}\right) + \alpha S + \eta\phi(I)F_aR - \xi(1 - \phi(I))R, \\[4pt]
\partial_t c = d_c\Delta c + \mathcal{Q}(c), \qquad \partial_c\mathcal{Q}(c) < 0,
\end{cases}
\end{align}
with the remaining components $(I,P,F_a)$ evolving as in \eqref{eq:system}. 
The condition $\partial_c\mathcal{Q}<0$ encodes signal decay/relaxation and provides a damping mechanism that competes against chemotactic focusing.

In the single-dose setting, Proposition~\ref{prop:IPA_bounds} implies $I(\cdot,t)\to 0$ exponentially, hence $\delta(I)$ and $\phi(I)$ vanish asymptotically. Therefore, the mode-wise stability question near the long-time homogeneous state reduces to an autonomous $(S,R)$ kinetics, and we isolate the transport mechanism via the templates \eqref{eq:oneway_logis_damping}–\eqref{eq:oneway}. Appendix~\ref{app:proof_corollary_oneway_damping} details rigorous mathematical derivations for such reduction and justify the drug-free dynamics in \eqref{eq:oneway_logis_damping}–\eqref{eq:oneway} as a chemotaxis extension of \eqref{eq:system}.
To emphasize transferability beyond the specific tumor--drug--stroma nonlinearities, we also consider the damped unidirectional templates
\begin{align}
\label{eq:oneway_logis_damping}
\begin{cases}
\partial_t S = d_S\Delta S - \chi_S\nabla \cdot (S\nabla c) + \lambda_S S\left(1 - \tfrac{S+R}{K} \right) + f_S(S,R), \\[2pt]
\partial_t R = d_R\Delta R - \chi_R\nabla \cdot (R\nabla c) + \lambda_R R\left(1 - \tfrac{S+R}{K} \right) + f_R(S,R), \\[2pt]
\partial_t c = d_c\Delta c + \mathcal{Q}(c), \qquad \partial_c\mathcal{Q}(c) < 0,
\end{cases}
\end{align}
and its abstract counterpart
\begin{align}
\label{eq:oneway}
\begin{cases}
\partial_t S = d_S\Delta S - \chi_S\nabla \cdot (S\nabla c) + f_S(S,R), \\[2pt]
\partial_t R = d_R\Delta R - \chi_R\nabla \cdot (R\nabla c) + f_R(S,R), \\[2pt]
\partial_t c = d_c\Delta c + \mathcal{Q}(c), \qquad \partial_c\mathcal{Q}(c) < 0.
\end{cases}
\end{align}
Here $f_S,f_R\in C^1(\mathbb R^2)$ denote the non-logistic reaction terms
(switching, drug, and stromal coupling). In the single-dose washout limit used
throughout, they reduce to the switching kinetics
\begin{equation}
\label{eq:fSfR_explicit}
f_S(S,R)=-\alpha S+\xi R,\qquad f_R(S,R)=\alpha S-\xi R,
\end{equation}
so that \eqref{eq:oneway_logis_damping} coincides with the reduced kinetics
\eqref{eq:dynamical}; in the abstract templates
\eqref{eq:oneway_logis_damping}, \eqref{eq:oneway} and \eqref{eq:twoway}, $f_S,f_R$ may be any $C^1$
kinetics with a stable interior equilibrium, of which
\eqref{eq:fSfR_explicit} is the representative instance.

\begin{assumption}[\textbf{Growth conditions for global existence of the unidirectional templates}]
\label{ass:reaction_growth}
For the unidirectional templates \eqref{eq:oneway_logis_damping} and
\eqref{eq:oneway}, in which the signal $c$ is decoupled from $(S,R)$, we assume the
kinetics $f_S,f_R\in C^1(\mathbb R_{\ge 0}^2)$ are quasi-positive,
\[
f_S(0,R)\ge 0,\quad f_R(S,0)\ge 0\qquad\text{for all }S,R\ge 0,
\]
and admit a linear growth bound: there is a constant $L\ge 0$ with, for all
$S,R\ge 0$,
\begin{equation}
\label{eq:linear_growth_bound}
\lambda_S S\Bigl(1-\tfrac{S+R}{K}\Bigr)+f_S(S,R)\le L\,(1+S+R),
\qquad
\lambda_R R\Bigl(1-\tfrac{S+R}{K}\Bigr)+f_R(S,R)\le L\,(1+S+R)
\end{equation}
(for \eqref{eq:oneway}, in which the logistic terms are absent, the same bound is
imposed on $f_S,f_R$ directly). Quasi-positivity guarantees forward invariance of
the nonnegative cone (as in Theorem~\ref{thm:positivity}); the bound
\eqref{eq:linear_growth_bound} yields, via a scalar comparison argument
(Proposition~\ref{prop:global_oneway}), global-in-time $L^\infty$ control.

We emphasize that this assumption is not imposed on, and does not by itself
control, the bidirectional template \eqref{eq:twoway}: there the taxis is a
cross-diffusion term whose principal part $\mathcal D_{\mathrm{diff}}-\mathcal H$ may
fail to be parabolic (Regime~III in Theorem~\ref{thm:twoway_classification}), so global existence is regime-dependent and is
addressed separately in Remark~\ref{rmk:wp_extension}.
\end{assumption}

\medskip
\noindent\textbf{Bidirectional (closed-loop) feedback.}
To model settings in which the populations actively shape the guiding signal, we consider feedback systems in which the signal landscape depends on $(S,R)$. 
A convenient abstraction---sufficient for the linearized mechanism classification developed later---is the closed-loop class
\begin{align}
\label{eq:twoway}
\begin{cases}
\partial_t S = d_S\Delta S - \chi_S\nabla\cdot\!\bigl(S\nabla g(S,R)\bigr) + f_S(S,R), \\[2pt]
\partial_t R = d_R\Delta R - \chi_R\nabla\cdot\!\bigl(R\nabla g(S,R)\bigr) + f_R(S,R),
\end{cases}
\end{align}
where $g(S,R)$ represents a feedback-generated signal landscape and, as in
\eqref{eq:fSfR_explicit}, $f_S,f_R$ denote general stable kinetics. We stress
that \eqref{eq:twoway} is written as an abstract closed-loop class purely
to state the linearized classification (\Cref{thm:twoway_classification}) at
maximal generality: the logistic competition and switching damping of
\eqref{eq:system} are subsumed into $f_S,f_R$ and are not removed. The
concrete closed-loop model actually simulated, \eqref{eq:feedback_model_c},
carries the full logistic kinetics of \eqref{eq:system} explicitly; thus
\eqref{eq:oneway_damping} and \eqref{eq:feedback_model_c} are the concrete
open- and closed-loop models with identical $(S,R)$ kinetics, differing only in
whether the signal $c$ is produced by $S$.

\medskip
\noindent\textbf{Well-definedness assumptions.}
In the unidirectional setting \eqref{eq:oneway_damping}--\eqref{eq:oneway}, the parabolic evolution of $c$ ensures $\nabla c$ is well-defined (in the weak sense) whenever $c\in L^2(0,T;H^1(U))$.
In the bidirectional abstraction \eqref{eq:twoway}, the fluxes are written directly in terms of $\nabla g(S,R)$. 
Accordingly, we assume $g\in C^1(\mathbb{R}^2)$ and work near sufficiently regular solutions so that the compositions $g(S,R)$ and $\nabla g(S,R)$ are meaningful; the stability analysis in Section~\ref{sec:chemotaxis} is carried out at smooth homogeneous equilibria, where these requirements are satisfied.
We note that, under the growth condition of Assumption~\ref{ass:reaction_growth},
global well-posedness carries over from Theorem~\ref{thm:global} to the
unidirectional class modulo the lower-order taxis perturbation; the concrete
model \eqref{eq:oneway_damping} satisfies this condition automatically. For the
bidirectional class, global existence holds only in the strongly parabolic (or
regularized) regime. The linearized classification of
Section~\ref{sec:chemotaxis} requires only local well-posedness near a smooth
equilibrium; see Remark~\ref{rmk:wp_extension}.

\section{Base Model Analysis}\label{sec:base_analysis}
\subsection{Invariant region and nonnegativity}
\label{subsec:invariant_positivity}

We first record the forward-invariant region underlying the base hybrid system \eqref{eq:system}. Many variants about standard positivity/invariance results for semilinear parabolic systems with homogeneous Neumann boundary conditions appear in the literature; see, e.g., \citep{amann1995linear,smoller2012shock}.
Since the state variables represent densities and concentrations, the physically meaningful phase space is the nonnegative cone. 
The result below establishes preservation of nonnegativity for all components under the standing Neumann boundary condition \eqref{eq:neumann_bc}.

\begin{theorem}[\textbf{Nonnegativity}]
\label{thm:positivity}
Let $U\subseteq\mathbb{R}^N$ ($N\in\{1,2,3\}$) be a bounded domain with $C^\infty$ boundary, and consider the base system \eqref{eq:system} under homogeneous Neumann boundary conditions
\[
\partial_{\mathrm n}S=\partial_{\mathrm n}R=\partial_{\mathrm n}I=0 \qquad \text{on }\partial U\times(0,\infty).
\]
Assume that all parameters in \eqref{eq:system} are nonnegative constants and that the initial data satisfy
\[
S_0,R_0,I_0,P_0,F_{a,0}\ge 0 \quad \text{a.e.\ in }U,
\qquad 
(S_0,R_0,I_0,P_0,F_{a,0})\in (L^\infty(U))^5.
\]
Then any classical solution $\mathbf{u}=(S,R,I,P,F_a)$ of \eqref{eq:system} on its maximal interval of existence $[0,T_{\max})$ remains nonnegative:
\[
S(x,t),\,R(x,t),\,I(x,t),\,P(x,t),\,F_a(x,t)\ge 0
\qquad \text{for all }(x,t)\in \overline U\times[0,T_{\max}).
\]
\end{theorem}

\begin{proof}
\noindent\textbf{Step 1: Quasi-positivity.}
Let $\mathcal G=(G_1,\dots,G_5)^{\mathsf T}$ be as in \eqref{eq:base_vector_form} with $\mathbf u=(S,R,I,P,F_a)^{\mathsf T}$. 
We verify that $\mathcal G$ is quasi-positive on $\mathbb{R}^5_{\ge 0}$: for each component $i$, if $u_i=0$ and $\mathbf u\ge 0$, then $G_i(\mathbf u)\ge 0$. 
Indeed, for $\mathbf u\ge 0$,
\begin{align*}
G_1\big|_{S=0} &= \xi(1-\phi(I))R \ge 0,\\
G_2\big|_{R=0} &= \alpha S \ge 0,\\
G_3\big|_{I=0} &= 0,\\
G_4\big|_{P=0} &= \beta(1-\phi(I))F_a \ge 0,\\
G_5\big|_{F_a=0} &= \theta\phi(I)P \ge 0,
\end{align*}
since $\phi(I)\in(0,1)$ and all parameters are nonnegative.

\smallskip
\noindent\textbf{Step 2: Positivity of the diffusing components $(S,R,I)$.}
Let $\mathbf{u}=(S,R,I)$. We define the maximal interval of nonnegativity as
\begin{equation*}
    \tau\coloneqq \sup\left\{t\in [0,T_{\max}): S(\cdot,s), R(\cdot,s), I(\cdot,s)\ge 0 \text{ on } \bar U \text{ for all } s\in [0,t] \right\}.
\end{equation*}
We claim that $\tau=T_{\max}$. 
Suppose, for the sake of contradiction, that $\tau<T_{\max}$.
By the definition of $\tau$ and continuity, we have:
\begin{enumerate}[label=(\roman*)]
    \item $S(x,t) \ge 0$, $R(x,t)\ge 0$, and $I(x,t)\ge 0$ for all $x\in \bar U$ and $t\in [0,\tau]$.
    \item There exists a component (without loss of generality, assume $S$) and a point $x^*\in \bar U$ such that $S(x^*,\tau)=0$.
\end{enumerate}
At the point $(x^*,\tau)$, $S$ attains a global spatial minimum (value $0$). If $x^*\in U$, then $\Delta S(x^*,\tau)\ge 0$ by the standard necessary condition for a minimum. We next show that the same sign condition holds if $x^*\in \partial U$. The homogeneous Neumann boundary condition,
\[
\partial_{\mathrm n} S(x^*,\tau) = \nabla S(x^*,\tau)\cdot \mathrm n(x^*) = 0,
\]
implies that the gradient of $S$ at $x^*$ lies in the tangent space $T_{x^*}\partial U$, where $\mathrm n(x^*)$ is the unit outward normal vector of $\partial U$ at $x^*$.
Taylor expansion about $S$ at $x^*$ gives
\begin{equation*}
    S(x,\tau) = S(x^*,\tau) + \nabla S(x^*,\tau)\cdot (x-x^*) + \frac{1}{2} (x-x^*)^\top \mathrm{H}S(x^*,\tau) (x-x^*) + o(\lVert x-x^*\rVert ^2), 
\end{equation*}
where $\mathrm{H}S(x^*,\tau)$ is the Hessian matrix of $S$ at $x^*$.
The minimum at $x^*$ implies that either 
\begin{equation}\label{eq:minimum_condition1}
\nabla S(x^*,\tau)\cdot (x-x^*) \ge 0,
\end{equation}
or, 
\begin{equation}\label{eq:minimum_condition2}
\nabla S(x^*,\tau)\cdot (x-x^*) = 0 \text{ and } \frac{1}{2} (x-x^*)^\top \mathrm{H}S(x^*,\tau) (x-x^*)\ge 0
\end{equation}
Let $x\in U$, we define a projection operator:
\begin{align*}
    T_0: U &\to T_{x^*}\partial U \\[4pt]
       x &\mapsto (x-x^*) + \Big(\mathrm n(x^*)\cdot (x-x^*)\Big)\, \mathrm{n}(x^*).
\end{align*}
$T_0(x)$ is the projection of $x-x^*$ onto the tangent space $T_{x^*}\partial U$. Since $U$ is a bounded domain with $C^\infty$ boundary $\partial U$, it is easy to see the linear span of $T_0(x)$ for $x\in U$ is the whole tangent space:
\begin{equation*}
    \operatorname{span}_{\mathbb{R}} \{T_0(x): x\in U\} = T_{x^*}\partial U.
\end{equation*}
\eqref{eq:minimum_condition1} and $\nabla S(x^*,\tau)\in T_{x^*}\partial U$ ensure that $\nabla S(x^*,\tau) = 0$. \eqref{eq:minimum_condition2} ensures that $\mathrm{H}S(x^*,\tau)$ is positive semidefinite. Therefore, we have for $x^*\in \partial U$,
\[
\Delta S(x^*,\tau) \ge 0.
\]
Going back to the $S$-equation at $(x^*,\tau)$:
\[
\partial_t S(x^*,\tau) = d_S\Delta S(x^*,\tau) + G_1(S = 0) \ge 0.
\]
Therefore, $S$ cannot become negative in the immediate future $t>\tau$. This contradiction implies $\tau=T_{\max}$.

\smallskip
\noindent\textbf{Step 3: Positivity of the stromal ODE subsystem $(P,F_a)$.}
For each fixed $x\in U$ the stromal variables satisfy the linear ODE system
\begin{equation}\label{eq:PA_ODE_pointwise}
\frac{d}{dt}
\begin{pmatrix}
P(x,t)\\ F_a(x,t)
\end{pmatrix}
=
B(I(x,t))
\begin{pmatrix}
P(x,t)\\ F_a(x,t)
\end{pmatrix},
\qquad
B(I)=
\begin{pmatrix}
-\theta\phi(I) & \beta(1-\phi(I))\\
\theta\phi(I) & -\beta(1-\phi(I))
\end{pmatrix}.
\end{equation}
Note that $B(I)$ is a Metzler matrix for every $I\ge0$ (its off-diagonal entries are nonnegative).

Fix $x\in U$ and write $p(t)=P(x,t)$, $a(t)=F_a(x,t)$.
If $p(t_0)=0$ at some time $t_0$ with $a(t_0)\ge0$, then from \eqref{eq:PA_ODE_pointwise},
\[
p'(t_0)=\beta(1-\phi(I(x,t_0)))\,a(t_0)\ge 0,
\]
so $p$ cannot decrease below $0$ at $t_0$.
Similarly, if $a(t_0)=0$ with $p(t_0)\ge0$, then
\[
a'(t_0)=\theta\phi(I(x,t_0))\,p(t_0)\ge 0,
\]
so $a$ cannot decrease below $0$ at $t_0$.
Hence the nonnegative quadrant in $(p,a)$ is forward invariant, proving the claim.

Combining Steps 2 and 3 yields nonnegativity of all components on $\overline U\times[0,T_{\max})$.
\end{proof}

\subsection{A priori bounds for the drug and stromal subsystem}
\label{subsec:bounds_IPA}

The base hybrid model \eqref{eq:system} contains two structural features that provide robust a priori control on the maximal existence interval: 
(i) the drug field $I$ satisfies a decoupled linear parabolic equation with clearance, and 
(ii) in the hybrid regime, the stromal conversion subsystem $(P,F_a)$ is conservative and bounded. 
We record these bounds in a form used repeatedly in the global continuation and long-time asymptotics arguments.

\begin{proposition}[\textbf{Bounds for $(I,P,F_a)$}]
\label{prop:IPA_bounds}
Let $\mathbf u=(S,R,I,P,F_a)$ be a classical solution of \eqref{eq:system} on its maximal interval of existence $[0,T_{\max})$, under homogeneous Neumann boundary conditions for $(S,R,I)$ and with nonnegative initial data $\mathbf u_0\in (L^\infty(U))^5$.

\begin{enumerate}[label=(\roman*)]
\item \textbf{Drug decay and $L^\infty$ control.}
The drug satisfies
\begin{equation}
\label{eq:I_decay}
\|I(t)\|_{L^\infty(U)} \le e^{-\gamma_I t}\,\|I_0\|_{L^\infty(U)}
\qquad \text{for all } t\in[0,T_{\max}).
\end{equation}

\item \textbf{Pointwise conservation of the stromal pool.}
For all $(x,t)\in U\times[0,T_{\max})$,
\begin{equation}
\label{eq:PA_conservation}
P(x,t)+F_a(x,t)=P_0(x)+F_{a,0}(x).
\end{equation}
In particular,
\begin{equation}
\label{eq:PA_bounds_pointwise}
0\le P(x,t),\,F_a(x,t)\le P_0(x)+F_{a,0}(x)\qquad \text{for all }(x,t)\in U\times[0,T_{\max}).
\end{equation}

\item \textbf{Uniform $L^\infty$ boundedness of $(I,P,F_a)$.}
For all $t\in[0,T_{\max})$,
\begin{equation}
\label{eq:IPA_uniform_bound}
\|I(t)\|_{L^\infty(U)}+\|P(t)\|_{L^\infty(U)}+\|F_a(t)\|_{L^\infty(U)}
\le \|I_0\|_{L^\infty(U)}+\|P_0\|_{L^\infty(U)}+\|F_{a,0}\|_{L^\infty(U)}.
\end{equation}
\end{enumerate}
\end{proposition}

\begin{proof}
\noindent\textbf{(i) Drug decay.}
The drug equation is
\[
\partial_t I = d_I\Delta I-\gamma_I I \quad \text{in }U\times(0,T_{\max}),
\qquad \partial_{\mathrm n} I=0 \quad \text{on }\partial U\times(0,T_{\max}).
\]
Define $J(x,t):=e^{\gamma_I t}I(x,t)$. Then $J$ solves $\partial_t J=d_I\Delta J$ with homogeneous Neumann boundary condition, hence by the maximum principle,
\[
\|J(t)\|_{L^\infty(U)}\le \|J(0)\|_{L^\infty(U)}=\|I_0\|_{L^\infty(U)}.
\]
Multiplying back by $e^{-\gamma_I t}$ yields \eqref{eq:I_decay}.

\smallskip
\noindent\textbf{(ii) Conservation of $P+F_a$.}
In the hybrid regime, summing the $P$- and $F_a$-equations in \eqref{eq:system} gives $\partial_t(P+F_a)=0$ pointwise, which implies \eqref{eq:PA_conservation}. 
Together with nonnegativity (Theorem~\ref{thm:positivity}), this yields \eqref{eq:PA_bounds_pointwise}.

\smallskip
\noindent\textbf{(iii) Uniform $L^\infty$ bound.}
The estimate \eqref{eq:IPA_uniform_bound} follows by combining \eqref{eq:I_decay} with \eqref{eq:PA_bounds_pointwise} and taking $L^\infty(U)$ norms.
\end{proof}

\begin{remark}
The bounds in Proposition~\ref{prop:IPA_bounds} identify the key invariant structure used in the global continuation argument: $I$ is uniformly controlled and decays exponentially, while the stromal pool $P+F_a$ is conserved pointwise and hence $(P,F_a)$ remain uniformly bounded on $[0,T_{\max})$.
\end{remark}
\subsection{Well-posedness and global-in-time solutions}
\label{subsec:global_existence}

We next establish global well-posedness for the base hybrid system \eqref{eq:system} under homogeneous Neumann boundary conditions. 
The argument follows a standard semigroup framework for semilinear parabolic systems: local existence and uniqueness are standard; global continuation then follows once finite-time blow-up in the $L^\infty$ norm is excluded. 
The invariant-region bounds from Sections~\ref{subsec:invariant_positivity}--\ref{subsec:bounds_IPA}, in particular the a priori control of $(I,P,F_a)$ and comparison bounds for $(S,R)$, provide the key input for the continuation step.
To state the global result, we use the standard blow-up alternative for mild solutions in $L^\infty(U)$; we omit semigroup technicalities and refer to standard sources (e.g.\ \citep{amann1995linear,evans2022partial}) for the abstract framework.

\begin{theorem}[\textbf{Global existence of mild solutions and classical regularity}]\citep{rothe2006global}
\label{thm:global}
Let $U\subseteq\mathbb{R}^N$ ($N\in\{1,2,3\}$) be a bounded $C^\infty$ domain. 
Consider the base system \eqref{eq:system} under homogeneous Neumann boundary conditions for $(S,R,I)$ and assume nonnegative initial data
\[
\mathbf u_0=(S_0,R_0,I_0,P_0,F_{a,0})\in (L^\infty(U))^5,
\qquad \mathbf u_0\ge 0 \ \text{a.e.\ in }U.
\]
Then:
\begin{enumerate}[label=(\roman*)]
\item \textbf{Global mild well-posedness.}
There exists a unique global-in-time mild solution
\[
\mathbf u(\cdot,t)=(S(\cdot,t),R(\cdot,t),I(\cdot,t),P(\cdot,t),F_a(\cdot,t))\in (L^\infty(U))^5
\qquad\text{for all }t\ge 0.
\]

\item \textbf{Classical regularity under H\"older initial data.}
If, in addition,
\begin{align}
\label{eq:u_regularity}
(S_0,R_0,I_0)\in \bigl(C^{2+\alpha}(\overline U)\bigr)^3,
\qquad
(P_0,F_{a,0})\in \bigl(C^{\alpha}(\overline U)\bigr)^2,
\qquad 0<\alpha<1,
\end{align}
and the Neumann compatibility conditions hold,
\[
\partial_{\mathrm n}S_0=\partial_{\mathrm n}R_0=\partial_{\mathrm n}I_0=0
\qquad\text{on }\partial U,
\]
then the mild solution is classical. More precisely,
\[
S,R,I \in C^{2+\alpha,\,1+\alpha/2}(\overline{U}\times [0,\infty)),
\qquad 
P,F_a \in C^{\alpha,\,1+\alpha/2}(\overline{U}\times [0,\infty)),
\]
where for $P,F_a$ the time regularity is inherited from the pointwise ODE structure in the baseline hybrid regime.
\end{enumerate}
\end{theorem}

See \citep{ladyzenskaia1968linear} for the definition of these H\"older spaces.

\begin{proof}
Rothe et al. establish the unique mild solution for the semilinear parabolic equations on the maximal interval of existence $[0,T_{\max})$ with $T_{\max}\in (0,\infty]$ with classical properties stated in (ii), and the following alternative:
\begin{equation}\label{eq:alternative}
    \lim_{t\to T_{\max}^-} \lVert \mathbf{u}(t)\rVert_{L^{\infty}(U)} = \infty \quad \text{ if } T_{\max} < \infty. 
\end{equation}

We show that for every $T<T_{\max}$,
\[
\sup_{0\le t\le T}\|\mathbf{u}(t)\|_{L^\infty(U)}\le C_T \quad \text{with} \quad \lim_{T\to T_{\max}^-} C_T <\infty,
\]
closing the proof and ensuring $T_{\max}=\infty$.

\begin{enumerate}[label=(\roman*)] 
\item Uniform control of $(I,P,F_a)$.
By Proposition~\ref{prop:IPA_bounds}, $(I,P,F_a)$ are uniformly bounded on $[0,T_{\max})$. In particular,
\begin{equation}
\label{eq:Famx_def}
\|F_a(t)\|_{L^\infty(U)}\le F_{a,\max}:=\|P_0\|_{L^\infty(U)}+\|F_{a,0}\|_{L^\infty(U)}
\qquad \text{for all }t\in[0,T_{\max}).
\end{equation}

\item Comparison bounds for $(S,R)$.
By Theorem~\ref{thm:positivity}, $S,R,I,P,F_a\ge0$ on $\overline U\times[0,T_{\max})$. 
Using $0\le \phi(I)\le 1$ and $\delta(I)\ge0$, and dropping nonpositive terms in \eqref{eq:system}, we obtain the pointwise inequalities
\begin{align}
\label{eq:S_ineq}
\partial_t S-d_S\Delta S &\le \lambda_S S + \xi R,
\\[4pt]
\label{eq:R_ineq}
\partial_t R-d_R\Delta R &\le \alpha S + \bigl(\lambda_R+\eta F_{a,\max}\bigr)R.
\end{align}
Let $(\bar s(t),\bar r(t))$ solve the linear ODE system
\begin{equation}
\label{eq:SR_ODE_compare}
\frac{d}{dt}
\begin{pmatrix}
\bar s\\ \bar r
\end{pmatrix}
=
\begin{pmatrix}
\lambda_S & \xi\\[2pt]
\alpha & \lambda_R+\eta F_{a,\max}
\end{pmatrix}
\begin{pmatrix}
\bar s\\ \bar r
\end{pmatrix},
\qquad
\begin{pmatrix}
\bar s(0)\\ \bar r(0)
\end{pmatrix}
=
\begin{pmatrix}
\|S_0\|_{L^\infty(U)}\\[2pt] \|R_0\|_{L^\infty(U)}
\end{pmatrix}.
\end{equation}
Since $(\bar s,\bar r)$ are spatially constant, they define a supersolution pair for \eqref{eq:S_ineq}--\eqref{eq:R_ineq}. 
By the parabolic comparison principle under homogeneous Neumann boundary conditions, it follows that
\[
S(\cdot,t)\le \bar s(t),\qquad R(\cdot,t)\le \bar r(t)\qquad \text{for all }t\in[0,T].
\]
Because \eqref{eq:SR_ODE_compare} is linear with constant coefficients, $(\bar s,\bar r)$ exist globally and grow at most exponentially. 
Hence for each fixed $T<T_{\max}$ there is a constant $C_T<\infty$ such that
\[
\sup_{0\le t\le T}\bigl(\|S(t)\|_{L^\infty(U)}+\|R(t)\|_{L^\infty(U)}\bigr)\le C_T.
\]
\end{enumerate}

Combining (i) and (ii) yields $\sup_{0\le t\le T}\|\mathbf{u}(t)\|_{L^\infty(U)}<\infty$ for every $T<T_{\max}$. 
By the blow-up alternative for maximal mild solutions in $L^\infty(U)$, this excludes finite-time blow-up and therefore $T_{\max}=\infty$.
\end{proof}

\begin{remark}[\textbf{Mechanism behind global continuation}]
The continuation argument hinges on two structural damping features in \eqref{eq:system}: 
(i) the drug field $I$ is uniformly controlled and decays exponentially, and 
(ii) the stromal conversion subsystem $(P,F_a)$ is conservative and uniformly bounded in the baseline hybrid regime. 
These properties provide a priori control that prevents finite-time blow-up of the remaining components and supports global-in-time well-posedness.
\end{remark}

\begin{remark}[\textbf{Extension to the chemotaxis models}]
\label{rmk:wp_extension}
The continuation argument for Theorem~\ref{thm:global} extends to the chemotaxis
extensions of Section~\ref{subsec:chemotaxis_models}, but the type of extension
depends on the coupling directionality and on the growth of the kinetics; this
dependence is itself part of the mechanism classification.

\emph{Unidirectional models \eqref{eq:oneway_damping}--\eqref{eq:oneway}.}
The signal $c$ solves a parabolic equation decoupled from $(S,R)$ and damped
($\partial_c\mathcal Q<0$); hence, for classical solutions with $c_0\in
C^{2+\alpha}(\overline U)$, $c$ is smooth for $t>0$ with $\nabla c,\Delta c\in
L^\infty$ on $[0,T]$. Writing the taxis flux as
$-\chi_W\nabla\!\cdot(W\nabla c)=-\chi_W\nabla W\!\cdot\!\nabla c-\chi_W\,W\,\Delta c$,
it is a bounded drift plus a bounded zeroth-order coefficient, i.e.\ a lower-order
perturbation of the diffusion operator that does not affect the $L^\infty$
comparison. Under Assumption~\ref{ass:reaction_growth} the quasi-positivity
preserves the nonnegative cone (Theorem~\ref{thm:positivity}) and the linear
growth bound \eqref{eq:linear_growth_bound} yields a spatially homogeneous
supersolution pair exactly as in \eqref{eq:SR_ODE_compare}; the $L^\infty$
blow-up alternative therefore closes and global classical solutions exist. The
concrete biological model \eqref{eq:oneway_damping} satisfies
Assumption~\ref{ass:reaction_growth} automatically: the logistic terms are
dissipative ($\lambda_W W(1-\tfrac{S+R}{K})\le\lambda_W W$), the switching terms
are linear, and the stromal coupling $\eta\phi(I)F_aR$ has the bounded coefficient
$\eta\phi(I)F_a\le \eta F_{a,\max}$ (Proposition~\ref{prop:IPA_bounds}), which
together reproduce the linear supersolution \eqref{eq:SR_ODE_compare}. We
emphasize that without a growth restriction the claim genuinely fails: a
superlinear source (e.g.\ a term $+S^p$, $p>1$, in $f_S$) can produce finite-time
blow-up even when $\nabla c$ is bounded, and the comparison method does not apply. We refer the reader to Appendix~\ref{appsec:global_bound_unidirectional} for the complete proof.

\emph{Bidirectional model \eqref{eq:twoway} and its realization
\eqref{eq:feedback_model_c}.} Here $c$ is coupled to the populations, so
\eqref{eq:feedback_model_c} is a logistic Keller--Segel-type system and
well-posedness is regime-dependent by design. Under strong parabolicity
(Assumption~\ref{ass:strong_parabolicity}) the linearized transport operator is
uniformly parabolic and local well-posedness holds; in the well-posed (Regime~II)
setting, global boundedness follows once the flux-saturation/volume-filling
regularization of Section~\ref{subsec:numerical_bifeedback} is in force. When
strong parabolicity fails (Regime~III, $\mathrm{sym}(M)$ has a negative
eigenvalue), the unregularized system can blow up in finite time and is
Hadamard ill-posed (Theorem~\ref{thm:twoway_classification}(iii)); we therefore do
not assert an unconditional global-existence statement for the closed-loop model,
and this loss of well-posedness is exactly the aggregation signal that motivates
regularization.

\emph{Role in the analysis.} Finally, the abstract templates
\eqref{eq:oneway_logis_damping}, \eqref{eq:oneway} and \eqref{eq:twoway} enter the paper only through
the linearized mechanism classification of
Section~\ref{sec:chemotaxis}, which is local: it requires only local
well-posedness near a smooth homogeneous equilibrium (guaranteed whenever
$f_S,f_R,g\in C^1$), not global-in-time existence. The growth condition
\eqref{eq:linear_growth_bound} is needed for the global statements above, not for
the classification results.
\end{remark}

\subsection{Long-time reduction and global attractivity of the homogeneous equilibrium}
\label{subsec:long_time_reduction}

We now characterize the long-time behavior of the base hybrid model \eqref{eq:system} under open-loop, single-dose drug forcing. 
The key structural feature is that the drug field $I$ evolves independently as a linear parabolic equation with clearance and therefore decays exponentially. 
Since the $I$-dependent couplings enter the remaining dynamics only through $\delta(I)$ and $\phi(I)$, the washout of $I$ implies that these couplings vanish asymptotically. 
As a consequence, the dominant long-time phenotypic composition is organized under the $I\equiv 0$ limit, in which we can close the tumor subsystem as an autonomous ODE for $(S,R)$ under the homogeneous limit $d_S=d_R=0$.

\medskip
\noindent\textbf{Drug relaxation and asymptotic decoupling.}
By Proposition~\ref{prop:IPA_bounds}(i), the drug satisfies
\[
\|I(t)\|_{L^\infty(U)} \le e^{-\gamma_I t}\,\|I_0\|_{L^\infty(U)} \qquad \text{for all }t\ge 0,
\]
and hence $I(\cdot,t)\to 0$ exponentially as $t\to\infty$. 
Moreover, since $\phi(I)\to 0$ as $I\to 0$ and $\delta(I)=\delta_0 I/(I+K_I)$ satisfies $\delta(I)\sim (\delta_0/K_I)I$ as $I\to 0$, all $I$-mediated interaction terms in \eqref{eq:system} decay exponentially in time. 
In the baseline hybrid system \eqref{eq:system}, the stromal subsystem is uniformly bounded and the stromal pool is conserved pointwise (Proposition~\ref{prop:IPA_bounds}(ii)).

\medskip
\noindent\textbf{Interpretation.}
Under homogeneous Neumann boundary conditions, diffusion preserves spatial averages and tends to damp spatial inhomogeneities. 
When $I(\cdot,t)\to 0$ uniformly, the reaction terms converge uniformly to their $I=0$ limits, so the local kinetics at different locations become asymptotically aligned. 
This motivates studying the homogeneous $I\equiv 0$ reduction as the organizing long-time limit for phenotypic composition. 
(We do not claim a global PDE convergence result  here; rather, the reduction identifies the limiting kinetics and provides the reference equilibrium used in the subsequent linear Turing stability analysis.)

\medskip
\noindent\textbf{Reduced homogeneous dynamics after washout.}
In the homogeneous limit (and in the absence of sustained drug input), the population dynamics reduce to
\begin{align}
\label{eq:dynamical}
\begin{cases}
\dfrac{dS}{dt} = f_S(S,R)\coloneqq \lambda_S S\Bigl(1-\dfrac{S+R}{K}\Bigr)-\alpha S+\xi R,\\[6pt]
\dfrac{dR}{dt} = f_R(S,R)\coloneqq \lambda_R R\Bigl(1-\dfrac{S+R}{K}\Bigr)+\alpha S-\xi R.
\end{cases}
\end{align}
Besides the trivial equilibrium $\mathbb{K}_0=(0,0)$, this system has an interior coexistence equilibrium
\[
\mathbb{K}^*=(S^*,R^*)
=\left(\frac{\xi K}{\alpha+\xi},\,\frac{\alpha K}{\alpha+\xi}\right),
\qquad \alpha+\xi>0,
\]
which represents the long-term phenotypic composition selected after drug washout.

\begin{theorem}[\textbf{Global attractivity of the coexistence equilibrium}]
\label{thm:global_attract}
Assume $\alpha+\xi>0$. 
Then the equilibrium $\mathbb{K}^*$ of \eqref{eq:dynamical} is globally asymptotically stable in the positive quadrant.
\end{theorem}

\begin{proof}
Consider the Lyapunov function
\[
V(S,R)=(S+R-K)^2.
\]
Along solutions of \eqref{eq:dynamical}, a direct calculation gives
\begin{align}
\label{eq:V_der}
\dot V(S,R)=-2K^{-1}(S+R-K)^2(\lambda_S S+\lambda_R R)\le 0.
\end{align}
Moreover,
\[
\dot V(S,R)=0 \quad\Longleftrightarrow\quad S+R=K,
\]
since $\lambda_S S+\lambda_R R\ge 0$ in the positive quadrant. 
For each $r>0$, the sublevel set
\[
\Omega_r=\{(S,R):V(S,R)\le r\}
\]
is compact and positively invariant. 
Let
\[
E_r=\{(S,R)\in\Omega_r:\dot V(S,R)=0\}=\{(S,R):S+R=K\}.
\]
By LaSalle's invariance principle \citep{lasalle1960some,haddad2008nonlinear}, every solution starting in $\Omega_r$ approaches the largest invariant subset of $E_r$ as $t\to\infty$. 
On the invariant line $S+R=K$, the dynamics reduce to the scalar linear ODE
\[
\frac{dS}{dt}=\xi(K-S)-\alpha S= \xi K-(\alpha+\xi)S,
\]
which has the unique equilibrium $S^*=\xi K/(\alpha+\xi)$ and is globally attracting on $[0,K]$. 
Consequently, the largest invariant subset of $E_r$ is $\{\mathbb{K}^*\}$, and thus $(S(t),R(t))\to \mathbb{K}^*$ as $t\to\infty$.
\end{proof}

\Cref{fig:phase_portrait} illustrates the phase portrait of the reduced $(S,R)$ subsystem \eqref{eq:dynamical}. 
Vector fields point toward the unique steady state $\mathbb{K}^*$, showing the global stability of $\mathbb{K}^*$ and validating the prediction of \Cref{thm:global_attract}. 
The long-time reduction is consistent with a broader phenomenon in feedback-regulated compartmental models, where integrating the structured dynamics yields balanced laws for total populations and an autonomous ODE limit system whose equilibria organize stability behaviors \citep{liang2025global}.

\begin{figure}
    \centering
    \includegraphics[width=0.6\linewidth]{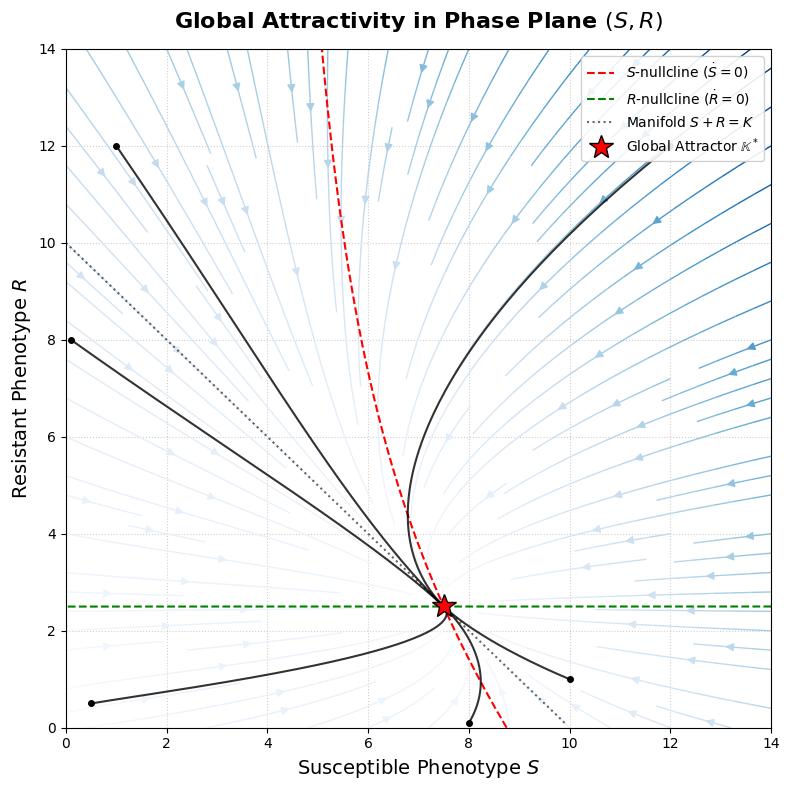}
    \caption{\textbf{Phase portrait of the reduced $(S,R)$ subsystem.} Red and green dashed lines depict the $S$- and $R$-nullclines, respectively. The black dashed line shows the one-dimensional invariant manifold $S+R=K$, where the $S$ and $R$ populations reach the carrying capacity $K$. The solid arrows demonstrate the direction and relative magnitude of the vector field, indicating the attractivity of the unique steady state $\mathbb K^*$ that is represented by the red pentagram. Representative trajectories starting from the black dots approach $\mathbb K^*$. Parameters: $K=10$, $\lambda_S=0.8$, $\lambda_R=0.4$, $\alpha=0.1$, $\xi=0.3$.}
    \label{fig:phase_portrait}
\end{figure}

\begin{remark}[\textbf{Phenotypic composition and a diagnostic ratio}]
\label{rmk:phy}
The explicit expression for $\mathbb{K}^*$ implies a simple diagnostic for the long-time phenotypic composition:
\[
\rho \coloneqq \frac{S^*}{R^*}=\frac{\xi}{\alpha}.
\]
Thus $\rho\gg 1$ corresponds to an $S$-dominated regime, $\rho\ll 1$ to an $R$-dominated regime, and $\rho\approx 1$ to balanced coexistence. 
This ratio will be used later to interpret how additional couplings (including chemotaxis and feedback) can shift the system between regimes.
\end{remark}

In summary, the open-loop drug field relaxes exponentially, so the $I$-mediated couplings in \eqref{eq:system} vanish asymptotically. 
The reduced homogeneous dynamics \eqref{eq:dynamical} therefore captures the long-time phenotypic composition selected after drug washout. 
This equilibrium provides the reference point for the Neumann linear stability and mechanism classification developed in Section~\ref{sec:chemotaxis}.

\subsection{No diffusion-driven instability in the damped base model (Turing stability)}
\label{subsec:turing_stability}

\begin{definition}[Homogeneous equilibrium]
\label{def:hom_eq}
In the following analysis, by a homogeneous equilibrium we mean a steady state
\[
(S^*,R^*,I^*,P^*,F_a^*)
\]
such that all components are spatially constant on $U$.
\end{definition}

A key baseline conclusion for \eqref{eq:system} is that its damping architecture does not support classical diffusion-driven (Turing-type) pattern generation after drug washout. 
Near the coexistence state selected by the $I\equiv 0$ limit, both tumor populations experience net negative feedback through logistic competition and switching, while diffusion under homogeneous Neumann boundary conditions acts as a smoothing mechanism. 
We formalize this by showing that no non-constant Neumann Laplacian mode can acquire positive growth rate for any choice of diffusion coefficients $d_S,d_R>0$. 

\begin{theorem}[\textbf{Turing stability of the reduced $(S,R)$ subsystem}]
\label{thm:turing}
Assume $\alpha+\xi>0$ and let $\mathbb{K}^*=(S^*,R^*)$ be the interior coexistence equilibrium of the reduced homogeneous system \eqref{eq:dynamical}. 
Consider the reaction--diffusion linearization of the $(S,R)$ subsystem under homogeneous Neumann boundary conditions. 
Let $\mu=\lambda_k\ge 0$ denote a Neumann Laplacian eigenvalue on $U$ (so that, on simple domains, $\mu=k^2$ with $k$ the wavenumber). 
Then for every $\mu>0$ all spectral growth rates satisfy $\Re \lambda(\mu)<0$. 
In particular, there is no diffusion-driven (Turing-type) instability: homogeneity cannot be destabilized by diffusion.
\end{theorem}

\begin{proof}
Linearizing \eqref{eq:dynamical} at $\mathbb{K}^*$ yields the Jacobian matrix $\mathcal{J}$ with
\[
\mathrm{tr}(\mathcal{J})
=\partial_S f_S(\mathbb{K}^*)+\partial_R f_R(\mathbb{K}^*)
=-\frac{\xi\lambda_S+\alpha\lambda_R}{\alpha+\xi}-\alpha-\xi<0,
\qquad
\det(\mathcal{J})=\xi\lambda_S+\alpha\lambda_R>0.
\]
Moreover, the diagonal entries satisfy $\partial_S f_S(\mathbb{K}^*)<0$ and $\partial_R f_R(\mathbb{K}^*)<0$.

For the reaction--diffusion linearization, each Neumann Laplacian eigenmode with eigenvalue $\mu=\lambda_k$ yields the dispersion matrix
\[
\mathcal{A}(\mu) = \mathcal{J} - \mu \operatorname{diag}(d_S,d_R)
= \begin{pmatrix}
    \partial_S f_S(\mathbb{K}^*)-\mu d_S & \partial_Rf_S(\mathbb{K}^*) \\[4pt]
    \partial_S f_R(\mathbb{K}^*) & \partial_Rf_R(\mathbb{K}^*)-\mu d_R
\end{pmatrix}.
\]
The characteristic polynomial for $\mathcal{A}(\mu)$ is
\[
\lambda^2-\mathrm{tr}(\mathcal{A}(\mu))\,\lambda+\det(\mathcal{A}(\mu))=0,
\]
or equivalently,
\[
\lambda^2+\lambda\Bigl(\mu(d_S+d_R)-\mathrm{tr}(\mathcal{J})\Bigr)+h(\mu)=0,
\]
where
\[
h(\mu)=d_Sd_R\mu^2-\bigl(d_S\,\partial_R f_R(\mathbb{K}^*)+d_R\,\partial_S f_S(\mathbb{K}^*)\bigr)\mu+\det(\mathcal{J}).
\]
Since $\mathrm{tr}(\mathcal{J})<0$ and $\partial_S f_S(\mathbb{K}^*)<0$, $\partial_R f_R(\mathbb{K}^*)<0$, we have
\[
\mu(d_S+d_R)-\mathrm{tr}(\mathcal{J})>0,
\qquad 
h(\mu)>0
\quad \text{for all }\mu>0.
\]
Therefore the Routh--Hurwitz criterion implies $\Re(\lambda(\mu))<0$ for all $\mu>0$, ruling out diffusion-driven destabilization.
\end{proof}

\begin{remark}[\textbf{Structural reason for no-Turing}]
\label{rmk:sign_structure}
The conclusion needs no dispersion computation: at $\mathbb K^*$ both diagonal
entries $\partial_S f_S(\mathbb K^*)<0$ and $\partial_R f_R(\mathbb K^*)<0$, so the
activator--inhibitor sign structure necessary for a diffusion-driven instability
is absent for any $d_S,d_R>0$. Diffusion-driven destabilization requires
$d_S\,\partial_R f_R+d_R\,\partial_S f_S>0$, which is impossible here.
\end{remark}

\begin{remark}[\textbf{Block structure for the full base model}]
\label{rmk:block_structure}
Theorem~\ref{thm:turing} is stated for the reduced $(S,R)$ reaction--diffusion subsystem (corresponding to the post-washout $I\equiv 0$ limit), but the same no-Turing conclusion is consistent with the full base model \eqref{eq:system}.

Fix a Neumann Laplacian eigenvalue $\mu=\lambda_k\ge 0$ and linearize \eqref{eq:system} about the homogeneous equilibrium branch with $I^*=0$, $(S^*,R^*)=\mathbb{K}^*$, and constant stromal pool $P^*+F_a^*=P_T$. 
After a permutation of variables consistent with the coupling topology (the drug equation is decoupled and the stromal subsystem does not feed back into $I$), the linearization can be written in block triangular form with diagonal blocks corresponding to $I$, $(P,F_a)$, and $(S,R)$.

The $I$-block contributes the eigenvalue
\[
\lambda_I(\mu)=-\gamma_I-d_I\mu<0,
\]
so the drug field produces no unstable spectrum. 
The $(P,F_a)$-block linearizes to a matrix with spectrum $\{0,-\beta\}$; the zero eigenvalue corresponds to the locally conserved quantity $P+F_a$ and does not generate spatial growth since $P$ and $F_a$ do not diffuse at this level (as further justified in Appendix~\ref{app:proof_corollary_oneway_damping}). 
Finally, the $(S,R)$-block yields the dispersion matrix $\mathcal{A}(\mu)$ in the proof of Theorem~\ref{thm:turing}. 
Hence the full linearization admits no diffusion-driven (Turing-type) instability under the standing assumptions.
\end{remark}

After the reductions established above, the effective stability problem is
two-dimensional (the $(S,R)$ block) for the base and unidirectional models, and
three-dimensional (the $(S,R,c)$ block) for the closed-loop model; the remaining
components are spectrally inert. This low effective dimension is a proven
consequence of drug washout and the conserved non-motile stroma, and it is what
permits a complete classification rather than a partial numerical study.

\section{Directionality--Damping Principle}\label{sec:chemotaxis}

\subsection{Mechanism statement}
\label{subsec:mechanism_statement}

Section~\ref{sec:base_analysis} established a baseline conclusion for the damped hybrid model \eqref{eq:system} under homogeneous Neumann boundary conditions: after a single-dose intervention, the drug field $I$ relaxes exponentially, the induced couplings vanish asymptotically, and the remaining dynamics are strongly dissipative. 
In particular, the post-washout reaction--diffusion core admits no diffusion-driven (Turing-type) bifurcation from homogeneity (Section~\ref{subsec:turing_stability}).

We now ask how directed transport modifies this baseline, and which structural features determine whether spatial heterogeneity can be generated from a homogeneous equilibrium. 
Chemotaxis introduces advective fluxes of the form $-\nabla\cdot(\chi_W W\nabla c)$, which bias migration of a population $W$ toward increasing values of a signal $c$. 
Such drift can transiently amplify spatial perturbations, but whether amplification produces a genuine pattern-forming instability is governed by two intertwined features:

\begin{enumerate}[label=(\roman*)]
\item \textbf{Directionality of coupling.}
In unidirectional (one-way) systems, the populations respond to a signal field but do not regulate it; the feedback loop required for self-reinforcing amplification is therefore absent. 
In bidirectional (two-way) systems, the signal depends on the populations (e.g.\ $c=g(S,R)$ or via a produced chemoattractant), closing a feedback loop and modifying the effective transport operator.

\item \textbf{Damping/relaxation structure.}
Logistic crowding, switching, and signal decay dissipate spatial perturbations and compete against chemotactic focusing. 
In the present model class, the base kinetics already contain strong damping (logistic competition and drug washout), providing a stringent test for whether directed transport can destabilize homogeneity.
\end{enumerate}

In the unidirectional setting, logistic damping together with signal relaxation implies that the homogeneous equilibrium cannot be destabilized in the Turing sense: one-way sensing is incapable of generating diffusion-driven patterns from homogeneity (\Cref{subsec:oneway}). 
By contrast, bidirectional feedback can destabilize intermediate wavelengths within the well-posed (strongly parabolic) regime (\Cref{subsec:twoway_parabolicity}), while exceeding the parabolicity threshold leads to Hadamard instability/aggregation and motivates biological regularization of the chemotactic flux (\Cref{subsec:twoway_classification}). 
The emphasis on generation from homogeneity under closed-loop coupling is consistent with recent tumor microenvironment models in which a chemotactic population both responds to and reinforces the guiding signal \citep{liu2025bidirectional}. 
The numerical experiments in Section~\ref{sec:numerical} illustrate both aspects: robust relaxation in the base and unidirectional regimes, and the restoration of finite-scale patterning only when minimal feedback closes the loop.

\subsection{Unidirectional coupling: no diffusion-driven pattern generation}
\label{subsec:oneway}

To emphasize transferability beyond the specific tumor--drug--stroma nonlinearities, we consider the unidirectional class \eqref{eq:oneway_logis_damping}--\eqref{eq:oneway}, where $c$ is a relaxing signal satisfying $\partial_c\mathcal Q<0$ and the populations $(S,R)$ experience dissipative kinetics. 
The condition $\partial_c\mathcal Q<0$ ensures that the signal relaxes.
The population senses a signal landscape but does not regulate it. 
In this open-loop setting, the feedback reinforcement required for diffusion-driven bifurcation is absent.

\begin{theorem}[\textbf{Unidirectional coupling cannot generate diffusion-driven patterns}]
\label{thm:oneway}
Consider \eqref{eq:oneway} with a unique interior equilibrium $\mathbf{u}^*=(S^*,R^*,c^*)$ and assume $\partial_c\mathcal{Q}<0$.
Let $\mathcal{J}_F$ be the Jacobian of $F=(f_S,f_R)$ at $(S^*,R^*)$,
\[
\mathcal{J}_F=
\begin{pmatrix}
a & b\\
c & d
\end{pmatrix},
\qquad
a=\partial_S f_S(\mathbf{u}^*),\ 
b=\partial_R f_S(\mathbf{u}^*),\ 
c=\partial_S f_R(\mathbf{u}^*),\ 
d=\partial_R f_R(\mathbf{u}^*).
\]
Assume the homogeneous kinetics are stable:
\[
a+d<0,
\qquad 
ad-bc>0.
\]
Then the homogeneous equilibrium is linearly stable with respect to all non-constant Neumann Laplacian modes provided either
\begin{enumerate}[label=(\roman*)]
\item $d\,d_S+a\,d_R<0$, \quad or
\item $4(ad-bc)\,d_Sd_R>(d\,d_S+a\,d_R)^2$.
\end{enumerate}
In particular, if $d_S=d_R$, stability holds for all modes and no diffusion-driven (Turing-type) bifurcation from homogeneity can occur.
\end{theorem}

\begin{proof}
Since the $c$-equation in \eqref{eq:oneway} is decoupled and satisfies $\partial_c\mathcal Q<0$, its linearization contributes only strictly negative mode-wise eigenvalues. 
Therefore mode-wise stability is determined by the $(S,R)$ subsystem. 
For each Neumann Laplacian eigenvalue $\mu=\lambda_k>0$, the dispersion matrix is
\[
\mathcal{A}(\mu)=\mathcal{J}_F-\mu
\begin{pmatrix}
d_S & 0\\
0 & d_R
\end{pmatrix}.
\]
We have $\mathrm{tr}(\mathcal{A}(\mu))=(a+d)-\mu(d_S+d_R)<0$ for all $\mu>0$. Moreover,
\[
\det(\mathcal{A}(\mu))
=(a-\mu d_S)(d-\mu d_R)-bc
= d_Sd_R\mu^2-(d\,d_S+a\,d_R)\mu+(ad-bc).
\]
Let $q(\mu):=\det(\mathcal{A}(\mu))$. Since $d_Sd_R>0$ and $q(0)=ad-bc>0$, it suffices to ensure $q(\mu)>0$ for all $\mu\ge 0$. 
If (i) holds, then the linear term in $q$ is nonnegative and hence $q(\mu)\ge q(0)>0$ for all $\mu\ge0$. 
If (ii) holds, then the discriminant of $q$ is negative and $q$ has no real roots; since $q(0)>0$ and $q$ is convex, it follows that $q(\mu)>0$ for all $\mu\ge0$. 
In either case, the Routh--Hurwitz criterion yields $\Re\lambda(\mu)<0$ for all $\mu>0$.

Finally, if $d_S=d_R=d$, then $\mathcal{A}(\mu)=\mathcal{J}_F-\mu d\,I$ and the spectrum satisfies $\lambda(\mu)=\lambda(0)-\mu d$ mode-wise. 
Thus all non-constant modes are strictly more stable than the homogeneous mode, and no diffusion-driven bifurcation can occur.
\end{proof}

\begin{corollary}[\textbf{Application to the unidirectional coupling system with logistic damping}]\label{cor:oneway_logis_damping}
System~\eqref{eq:oneway_logis_damping} is a concrete example of system~\eqref{eq:oneway}. 
Therefore the conclusions of Theorem~\ref{thm:oneway} apply to \eqref{eq:oneway_logis_damping} with the Jacobian evaluated for the damped kinetics. 
Specifically, define
\[
\tilde f_S(S,R)=\lambda_S S\Bigl(1-\frac{S+R}{K}\Bigr)+f_S(S,R),
\qquad
\tilde f_R(S,R)=\lambda_R R\Bigl(1-\frac{S+R}{K}\Bigr)+f_R(S,R),
\]
and let $\mathcal{J}_F$ in Theorem~\ref{thm:oneway} be computed using $\tilde f_S,\tilde f_R$ in place of $f_S,f_R$.
\end{corollary}

\begin{corollary}[\textbf{Application to the tumor--drug--stroma unidirectional model}]
\label{cor:oneway_damping}
Consider the unidirectional chemotaxis extension \eqref{eq:oneway_damping}, with $(I,P,F_a)$ evolving as in \eqref{eq:system} and with the signal $c$ modeled as a decaying open-loop field (i.e.\ $c$ receives no feedback from $(S,R)$). 
Assume the homogeneous equilibrium is linearly stable for the corresponding reduced $(S,R)$ subsystem. 
Then the homogeneous equilibrium cannot undergo diffusion-driven destabilization induced by the unidirectional chemotactic terms: unidirectional sensing does not generate Turing-type patterns from homogeneity.
\end{corollary}

\begin{remark}
Corollary~\ref{cor:oneway_damping} is the mechanism-level statement used in the remainder of the paper. 
The explicit coefficient computations for the concrete model \eqref{eq:oneway_damping} (including the fully expanded dispersion relation) are lengthy and are therefore deferred to Appendix~\ref{app:proof_corollary_oneway_damping}. 
The essential point is that in one-way coupling the signal landscape is not reinforced by the migrating populations, so the feedback loop required for diffusion-driven bifurcation is absent.
\end{remark}

\subsection{Bidirectional feedback: effective mobility and strong parabolicity}
\label{subsec:twoway_parabolicity}

We now turn to bidirectional coupling, where the guiding signal depends on the populations and therefore closes a feedback loop. 
In this setting, the chemotactic drift alters the transport operator itself and can create instability even when the homogeneous kinetics are stable. 
A key point---both analytically and for interpretation---is that the high-frequency behavior of the linearized problem is governed by an effective mobility matrix. 
This provides a clean separation between well-posed finite-wavelength (Turing-type) instability and Hadamard instability/aggregation arising from loss of parabolic smoothing.

We work with the bidirectional system \eqref{eq:twoway},
and assume it admits an interior equilibrium $\mathbf{u}^*=(S^*,R^*)$ satisfying
\[
f_S(S^*,R^*)=0,\qquad f_R(S^*,R^*)=0.
\]
Let $\mathcal{J}_F$ denote the Jacobian of $F=(f_S,f_R)$ at $\mathbf{u}^*$:
\[
\mathcal{J}_F=
\begin{pmatrix}
a & b\\
c & d
\end{pmatrix},
\qquad
a=\partial_S f_S(\mathbf{u}^*),\ 
b=\partial_R f_S(\mathbf{u}^*),\ 
c=\partial_S f_R(\mathbf{u}^*),\ 
d=\partial_R f_R(\mathbf{u}^*).
\]
Linearizing the chemotactic fluxes at $\mathbf{u}^*$ yields the feedback matrix
\[
\mathcal{H}(\mathbf{u}^*)=
\begin{pmatrix}
\chi_S\,\partial_S g(\mathbf{u}^*)\,S^* & \chi_S\,\partial_R g(\mathbf{u}^*)\,S^*\\[2pt]
\chi_R\,\partial_S g(\mathbf{u}^*)\,R^* & \chi_R\,\partial_R g(\mathbf{u}^*)\,R^*
\end{pmatrix}.
\]
Writing $\tilde{\mathbf{u}}=(S-S^*,\,R-R^*)^{\mathsf T}$, the linearization takes the form
\[
\partial_t \tilde{\mathbf{u}}
= (\mathcal{D}_{\mathrm{diff}}-\mathcal{H}(\mathbf{u}^*))\,\Delta \tilde{\mathbf{u}} + \mathcal{J}_F\tilde{\mathbf{u}},
\qquad 
\mathcal{D}_{\mathrm{diff}}=\begin{pmatrix} d_S & 0\\ 0 & d_R\end{pmatrix}.
\]
This motivates the definition of the effective mobility matrix
\begin{equation}
\label{eq:effective_mobility}
M:=\mathcal{D}_{\mathrm{diff}}-\mathcal{H}(\mathbf{u}^*)
=
\begin{pmatrix}
d_S-\chi_S\partial_S g(\mathbf{u}^*)S^* & -\chi_S\partial_R g(\mathbf{u}^*)S^*\\[2pt]
-\chi_R\partial_S g(\mathbf{u}^*)R^* & d_R-\chi_R\partial_R g(\mathbf{u}^*)R^*
\end{pmatrix}.
\end{equation}
For each Neumann Laplacian eigenvalue $\mu=\lambda_k\ge 0$, the mode-wise dispersion matrix is
\begin{equation}\label{eq:dispersion_matrix}
\mathcal{A}(\mu)=\mathcal{J}_F-\mu M,
\end{equation}
so that the growth rates satisfy $\det(\lambda I-\mathcal{A}(\mu))=0$. 
In particular, the asymptotic behavior as $\mu\to\infty$ is governed by the principal part $-\mu M$.

To correctly distinguish finite-wavelength pattern selection from high-frequency breakdown, we explicitly isolate the regime in which the linearized diffusion operator is uniformly parabolic.

\begin{assumption}[\textbf{Strong parabolicity of the effective mobility}]
\label{ass:strong_parabolicity}
There exists a constant $m_0>0$ such that the symmetric part of $M$ satisfies
\[
\mathrm{sym}(M):=\frac{M+M^{\mathsf T}}{2}\ \ge\ m_0 I
\qquad\text{(as quadratic forms on }\mathbb{R}^2\text{)}.
\]
Equivalently,
\[
\xi^{\mathsf T}\mathrm{sym}(M)\,\xi \ge m_0|\xi|^2\qquad \forall\,\xi\in\mathbb{R}^2.
\]
\end{assumption}

\begin{remark}[\textbf{Sufficient checks in the $2\times 2$ case}]
\label{rmk:parabolicity_conditions}
Assumption~\ref{ass:strong_parabolicity} is a standard sufficient condition ensuring that the operator associated with $M$ generates a well-posed forward-in-time evolution (high spatial frequencies are damped). 
Since $\mathrm{sym}(M)$ is symmetric, a convenient sufficient condition is positivity of its trace and determinant:
\[
\mathrm{tr}(\mathrm{sym}(M))>0,
\qquad 
\det(\mathrm{sym}(M))>0.
\]
When $M$ is symmetric (or approximately symmetric), strong parabolicity reduces to positive definiteness of $M$, which in the $2\times 2$ case can be checked by
\[
\mathrm{tr}(M)>0,
\qquad 
\det(M)>0.
\]
If Assumption~\ref{ass:strong_parabolicity} fails, i.e.\ $\mathrm{sym}(M)$ has a nonpositive eigenvalue, then high-frequency modes need not be damped. 
In that case, the dispersion relation associated with \eqref{eq:dispersion_matrix} may exhibit growth that increases without bound as $\mu\to\infty$, signaling Hadamard instability/aggregation rather than finite-wavelength Turing pattern selection.
\end{remark}

\subsection{Bidirectional feedback: stable, finite-band, and ill-posed regimes}
\label{subsec:twoway_classification}

We now state the instability classification for bidirectional feedback systems. 
The distinction hinges on the effective mobility matrix $M=\mathcal{D}_{\mathrm{diff}}-\mathcal{H}(\mathbf{u}^*)$ introduced in \eqref{eq:effective_mobility}. 
When $M$ is strongly parabolic (Assumption~\ref{ass:strong_parabolicity}), high-frequency modes are damped and any instability can only occur on a bounded band of intermediate wavelengths (finite-wavelength selection). 
When strong parabolicity fails, the linearized problem may exhibit Hadamard-type high-frequency amplification, corresponding to aggregation/collapse rather than finite-band Turing pattern selection.

\begin{theorem}[\textbf{Bidirectional feedback: stable, Turing-type, and aggregation/ill-posed regimes}]
\label{thm:twoway_classification}
Consider the bidirectional system \eqref{eq:twoway} on a bounded smooth domain $U$ under homogeneous Neumann boundary conditions. 
Let $\mathbf{u}^*=(S^*,R^*)$ be an interior equilibrium. 
Let $\mathcal{J}_F$ be the reaction Jacobian at $\mathbf{u}^*$ and let $M=\mathcal{D}_{\mathrm{diff}}-\mathcal{H}(\mathbf{u}^*)$ be the effective mobility matrix, as in \eqref{eq:effective_mobility}. 
For each Neumann Laplacian eigenvalue $\mu=\lambda_k\ge 0$, define the mode-wise dispersion matrix
\[
\mathcal{A}(\mu)=\mathcal{J}_F-\mu M,
\]
so that growth rates satisfy $\det(\lambda I-\mathcal{A}(\mu))=0$. 
Assume the homogeneous kinetics are stable:
\begin{equation}
\label{eq:kinetic_stability}
\mathrm{tr}(\mathcal{J}_F)<0,\qquad \det(\mathcal{J}_F)>0.
\end{equation}
Then the parameter space decomposes into the following regimes:
\begin{enumerate}[label=(\roman*)]
\item \textbf{Stable (well-posed and mode-wise stable).}
If Assumption~\ref{ass:strong_parabolicity} holds and, in addition,
\begin{equation}
\label{eq:modewise_stability}
\mathrm{tr}(\mathcal{A}(\mu))<0 \ \text{ and }\ \det(\mathcal{A}(\mu))>0
\qquad \text{for all }\mu\ge 0,
\end{equation}
then $\Re\lambda(\mu)<0$ for all modes $\mu\ge 0$. 
In particular, $\mathbf{u}^*$ is linearly stable and no diffusion-driven instability occurs.

\item \textbf{Finite-band instability (well-posed with finite-wavelength selection).}
Assume Assumption~\ref{ass:strong_parabolicity} holds. 
If there exists a mode $\mu_* >0$ such that either
\begin{equation}
\label{eq:turing_condition_general}
\mathrm{tr}(\mathcal{A}(\mu_*))>0
\qquad\text{or}\qquad
\det(\mathcal{A}(\mu_*))<0,
\end{equation}
then $\mathbf{u}^*$ is linearly unstable to spatial perturbations. 
Moreover, due to strong parabolicity, high-frequency modes are damped: there exists $\mu_1>0$ such that $\Re\lambda(\mu)<0$ for all $\mu\ge \mu_1$. 
Consequently, instability (if present) occurs on a bounded set of intermediate wavelengths, corresponding to finite-wavelength (Turing-type) pattern generation.

\item \textbf{Aggregation / ill-posedness (loss of parabolic smoothing).}
If Assumption~\ref{ass:strong_parabolicity} fails and $\mathrm{sym}(M)$ has a negative eigenvalue, then the principal part of the linearized operator contains a direction of anti-diffusion. 
In this case the dispersion relation exhibits Hadamard instability in the sense that
\[
\sup_{\mu>0}\Re\lambda(\mu)=+\infty,
\]
and the model must be interpreted as an aggregation/collapse limit of the unregularized continuum description \citep{joseph1990hadamard}. 
This regime should not be interpreted as finite-wavelength Turing pattern selection; rather it indicates that additional physical regularization (e.g.\ volume filling or flux limitation) is required for biological realism.
\end{enumerate}
\end{theorem}

\begin{proof}[Proof sketch]
The statement follows from the $2\times 2$ dispersion relation associated with $\mathcal{A}(\mu)=\mathcal{J}_F-\mu M$. 
Under \eqref{eq:kinetic_stability}, the homogeneous equilibrium is stable at $\mu=0$. 
If \eqref{eq:modewise_stability} holds for all $\mu\ge 0$, then the Routh--Hurwitz conditions imply $\Re\lambda(\mu)<0$ for every mode, giving (i).

Assume next that Assumption~\ref{ass:strong_parabolicity} holds. 
Then the principal part $-\mu M$ dominates as $\mu\to\infty$ and yields $\Re\lambda(\mu)\to -\infty$, so any instability must occur on a bounded interval of $\mu$; condition \eqref{eq:turing_condition_general} guarantees the existence of an unstable mode, giving (ii).

If $\mathrm{sym}(M)$ has a negative eigenvalue, then there exists a direction $\zeta\in\mathbb{R}^2$ such that $\zeta^{\mathsf T}\mathrm{sym}(M)\zeta<0$. 
Along the corresponding high-frequency modes, the principal contribution $-\mu M$ produces growth of order $+\mu$, hence $\sup_{\mu>0}\Re\lambda(\mu)=+\infty$, giving (iii).
\end{proof}

\begin{remark}[\textbf{Interpretation of regimes, parabolicity threshold, and the need for regularization}]
\label{rmk:regularization}
Regime (ii) corresponds to finite-wavelength selection: instability occurs only on an intermediate band of modes, while high frequencies remain damped by parabolic smoothing. 
By contrast, regime (iii) corresponds to Hadamard-type high-frequency amplification, in which the linearized growth rate can increase without bound as $\mu\to\infty$. 
The latter is an aggregation/collapse signal for the unregularized continuum model and should be interpreted as a prompt to introduce biologically meaningful regularization (e.g.\ flux limitation) rather than as classical Turing pattern formation. 
\Cref{fig:dispersion_classification} provide representative dispersion curves consistent with Theorem~\ref{thm:twoway_classification}. The classification result is parallel to the algebraic--spectral separation between different stability properties and bifurcation behaviors \citep{wang2026algebraic}.

\begin{figure}
    \centering
    \includegraphics[width=0.75\linewidth]{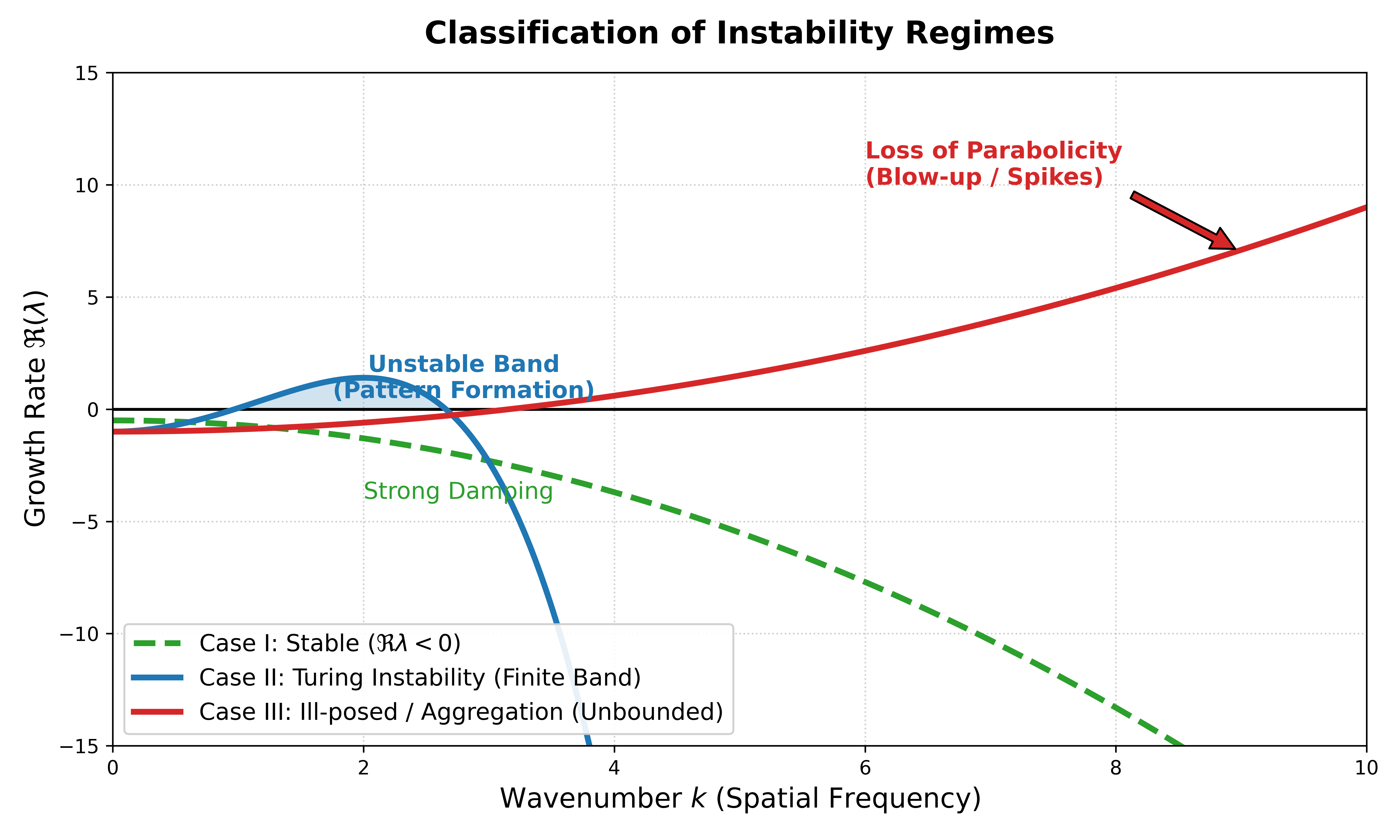}
    \caption{\textbf{Classification of stability regimes (\Cref{thm:twoway_classification}).} Plot of growth rate $\Re(\lambda)$ vs.\ wavenumber $k$. \textbf{Case I (green):} unconditional stability. \textbf{Case II (blue):} well-posed finite-band (Turing-type) instability. \textbf{Case III (red):} loss of parabolicity (growth rate diverges), indicating aggregation/ill-posedness.}
    \label{fig:dispersion_classification}
\end{figure}

A convenient diagnostic for the onset of aggregation/ill-posedness is loss of positivity of the diagonal effective diffusion terms in $M$. 
For example, parameter configurations that force
\begin{equation*}
d_S-\chi_S\,\partial_S g(\mathbf{u}^*)\,S^*<0
\qquad\text{and}\qquad
d_R-\chi_R\,\partial_R g(\mathbf{u}^*)\,R^*=0,
\end{equation*}
or
\begin{equation*}
d_S-\chi_S\,\partial_S g(\mathbf{u}^*)\,S^*=0
\qquad\text{and}\qquad
d_R-\chi_R\,\partial_R g(\mathbf{u}^*)\,R^*<0,
\end{equation*}
push the system toward regime (III) and should be interpreted as aggregation/ill-posedness rather than finite-band Turing patterning.

These conditions have a direct physical interpretation: they quantify when feedback-driven drift becomes comparable to, or exceeds, diffusive smoothing. 
In partially decoupled limits, the thresholds
\begin{equation}\label{eq:chemotaxis_threshold}
\chi_S^{c}=\frac{d_S}{\partial_S g(\mathbf{u}^*)\,S^*},
\qquad 
\chi_R^{c}=\frac{d_R}{\partial_R g(\mathbf{u}^*)\,R^*}
\end{equation}
mark the point at which taxis-induced focusing can dominate diffusion.

Importantly, exceeding such a threshold in an unregularized continuum model should not automatically be interpreted as biologically smooth ``pattern formation''. 
When feedback becomes strong enough to violate strong parabolicity of the effective mobility matrix $M=\mathcal{D}_{\mathrm{diff}}-\mathcal{H}(\mathbf{u}^*)$ (regime (III) in \Cref{thm:twoway_classification}), the linearized dispersion relation can exhibit unbounded high-frequency growth, signaling aggregation/collapse (Hadamard-type instability) rather than finite-wavelength Turing banding. 
From the standpoint of tissue physics, this corresponds to an ``over-attractive'' regime where cells (or mass) would concentrate toward delta-like peaks in the absence of additional constraints, which is incompatible with finite cell volume and crowding.
\end{remark}

\section{Numerical experiments: reproducible evidence for the mechanism}
\label{sec:numerical}

\subsection{Numerical scheme and implementation details}
\label{subsec:numerical_scheme}

This section provides reproducible numerical evidence supporting the mechanism classification developed in Section~\ref{sec:chemotaxis}. 
Our goals are to (i) corroborate relaxation toward homogeneity in the base and unidirectional regimes (Regime~I and Regime~I$'$), 
(ii) demonstrate that sustained finite-wavelength structure is recovered only under bidirectional feedback in a well-posed regime (Regime~II), 
and (iii) illustrate numerical breakdown consistent with the aggregation/ill-posed regime when strong parabolicity is violated (Regime~III).

We discretize the coupled system on a uniform Cartesian grid over $U=[0,1]^2$ with mesh size $h$ and time step $\Delta t$. 
For each diffusing component, we use a semi-implicit Crank--Nicolson update for the diffusion term and treat all remaining terms explicitly (reaction and, when present, taxis). 
Homogeneous Neumann boundary conditions are imposed by ghost-point reflection, ensuring zero normal flux at the discrete level.

In the baseline hybrid system \eqref{eq:system}, only $(S,R,I)$ are advanced by the diffusion step. 
In chemotaxis experiments, the taxis term requires a well-defined signal gradient; 
accordingly, the signal variable $c$ used in the taxis flux is evolved by a parabolic equation and the same semi-implicit diffusion update is applied to $c$ as well. 
In the unidirectional experiments we take $\mathcal Q(c)=-\rho_c\,c$, so that $\partial_c\mathcal Q(c)=-\rho_c<0$.
Unless otherwise stated, we take $N_x=N_y=51$ (so $h=1/(N_x-1)$) and $\Delta t=10^{-2}$. 
We use the time horizon $T_{\mathrm{final}}=5$ for Regimes~I and I$'$. 
For the feedback simulations in Regimes~II--III we use a longer horizon $T_{\mathrm{final}}=50$ to allow the feedback-driven structures to develop and saturate.

All simulations are initialized near the relevant homogeneous equilibrium by i.i.d.\ unbiased perturbations
\[
W(x_i,y_j,0)=W^*+\varepsilon\,\rho_{i,j},
\qquad 
\rho_{i,j}\sim\mathcal{U}(-1,1)\ \text{i.i.d.},
\]
applied to each diffusing component present in the model under study.
In particular, for the base system \eqref{eq:system} we perturb $W\in \{S,R,I\}$, 
whereas for chemotaxis extensions we additionally perturb the signal field $c$.
Throughout this section we fix $\varepsilon=10^{-3}$. 
Since $I^*=0$ in the washout equilibrium branch, the perturbation may produce negative values in the drug field; we enforce the physical constraint by truncating $I$ to be nonnegative after each update. 
For numerical robustness, we also truncate any negative values of the density variables $(S,R)$ and the signal $c$ to zero after each explicit update.

For the non-diffusing stromal variables $(P,F_a)$, we enforce the homogeneous pool assumption $P_0(x,y)+F_{a,0}(x,y)\equiv P_T=1$ by setting $(P,F_a)(x,y,0)=(P^*,F_a^*)=(1,0)$.

In Regime~II simulations, we employ a flux-saturated (velocity-limited) taxis to prevent non-physical collapse; the explicit form is given in \Cref{subsec:numerical_bifeedback}. 
Baseline parameter values used in the numerical experiments are listed in \Cref{tab:parameters}. Parameters that are used in a subset of experiments are indicated in the table.

\begin{table}[htbp]
\centering
\caption{Baseline parameter values used in the numerical experiments (units are given in Table~\ref{tab:parameter_units}). Parameters marked ``(uni)'' are used only in the unidirectional chemotaxis experiments; parameters marked ``(fb)'' are used only in the feedback experiments.}
\label{tab:parameters}
\begin{tabular}{@{} l c l c @{}}
\toprule
Parameter        & Value       & Parameter       & Value       \\
\midrule
$d_S$            & $10^{-3}$   & $d_R$           & $10^{-3}$   \\
$d_I$            & $5$         & $d_c$           & $5\times 10^{-3}$ \\
$\lambda_S$      & $0.5$       & $\lambda_R$     & $0.5$       \\
$K$              & $2.0$       & $\alpha$        & $0.1$       \\
$\xi$            & $0.1$       & $\eta$          & $0.2$       \\
$\theta$         & $1$         & $\beta$         & $0.5$       \\
$\gamma_I$       & $1$         & $\rho_c$        & $1.6$       \\
$\delta(I)$      & $\delta_0\,\dfrac{I}{I+K_I}$ 
                & $\phi(I)$ & $\tanh(\kappa_\phi I)$ \\
$\delta_0$       & $0.5$       & $K_I$           & $0.5$       \\
$\kappa_\phi$    & $5$         &                 &             \\
$\chi_S$ (uni)   & $0.5$       & $\chi_R$ (uni)  & $0.5$       \\
$\chi_S'$ (fb)   & $0.5$       & $\kappa_c$ (fb) & $0.8$       \\
$\alpha_c$ (fb)  & $100$       & $\varepsilon$   & $10^{-3}$   \\
\bottomrule
\end{tabular}
\end{table}

\subsection{Base model: relaxation to homogeneity (Regime I)}
\label{subsec:numerical_base_regimeI}

We first corroborate the relaxation-to-homogeneity behavior of the base hybrid model under homogeneous Neumann boundary conditions. 
Starting from unbiased perturbations about the homogeneous equilibrium, diffusion smooths spatial deviations while the damped reaction terms drive trajectories toward the equilibrium branch selected after washout. 
\Cref{fig:homogeneous} shows representative snapshots of $S$ at $t=0,\,2.5,\,5.0$, illustrating progressive homogenization; by $T=5.0$ the remaining spatial variation is at the level of numerical microfluctuations.

To quantify the decay of spatial heterogeneity, we track the $L^2$-deviation norms
\[
E_W(t):=\|W(\cdot,t)-W^*\|_{L^2(U)},
\qquad W\in\{S,R,I\}.
\]
As shown in \Cref{fig:norm_evolution} (blue solid curve for $E_S$), the deviation decays approximately exponentially after a short transient period, consistent with dissipativity and with the absence of diffusion-driven pattern generation in Regime~I.

\begin{figure}[htbp]
\centering
\includegraphics[width=\textwidth]{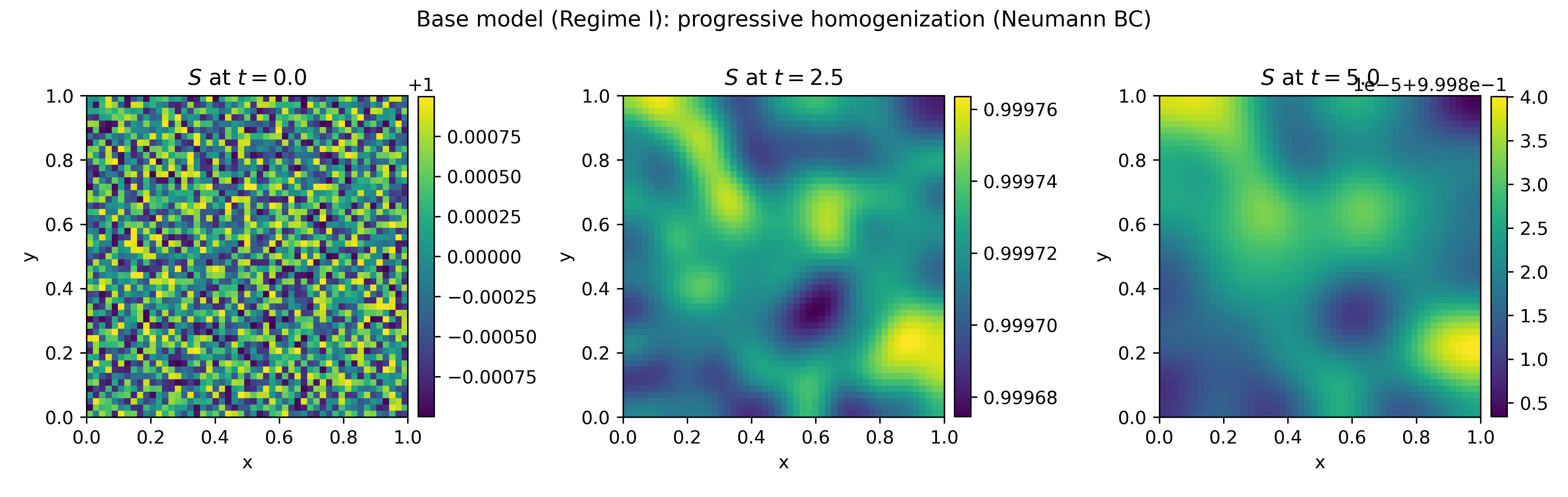}
\caption{\textbf{Base model: progressive homogenization.}
Spatial snapshots of $S$ at $t=0,\,2.5,\,5.0$. Starting from unbiased random perturbations, the field smooths out over time; the colorbar at $t=5.0$ indicates deviations of order $10^{-4}$ relative to the mean.}
\label{fig:homogeneous}
\end{figure}

\begin{figure}[htbp]
\centering
\includegraphics[width=0.7\textwidth]{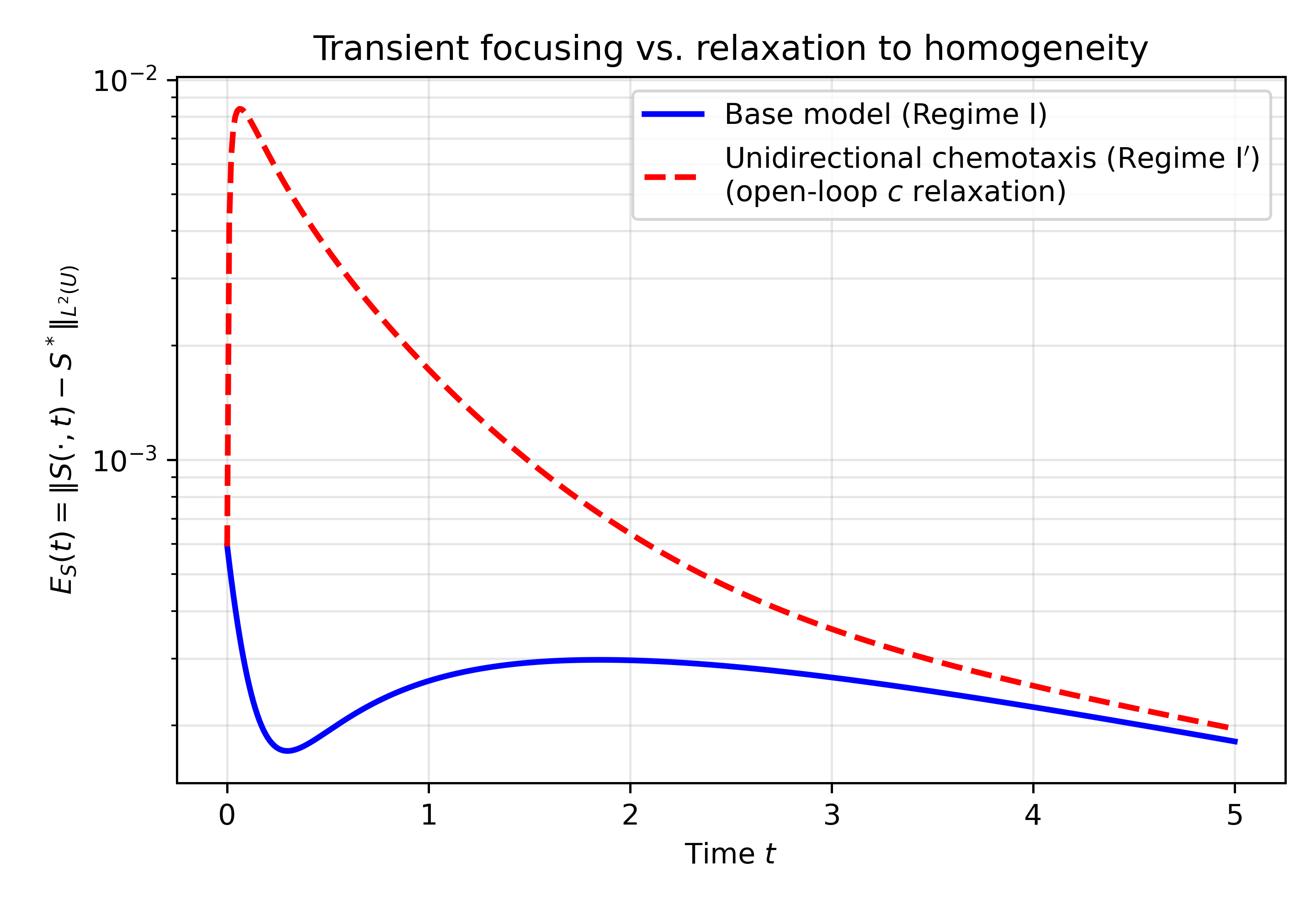}
\caption{\textbf{Quantification of convergence.}
Evolution of the $L^2$-deviation from equilibrium for the base model (blue solid) and the unidirectional chemotaxis model (red dashed). 
The base model exhibits monotone decay after a transient period. 
The unidirectional model shows a transient rise followed by decay, consistent with relaxation rather than sustained pattern generation.}
\label{fig:norm_evolution}
\end{figure}

\subsection{Unidirectional chemotaxis: transient focusing without generation (Regime I')}
\label{subsec:numerical_base_oneway}

We next test the ``no-generation'' prediction for unidirectional coupling (open-loop sensing). 
In this setting, directed motion can transiently sharpen existing heterogeneity by transporting mass toward prescribed signal maxima, but---in the absence of feedback from $(S,R)$ to the signal---it cannot sustain a diffusion-driven bifurcation from a homogeneous equilibrium.
We simulate the unidirectional chemotaxis extension \eqref{eq:oneway_damping} under homogeneous Neumann boundary conditions, using the same initialization protocol as in \Cref{subsec:numerical_scheme}. 
Here $S$ and $R$ drift up gradients of the signal field $c$, while $c$ evolves independently of $(S,R)$ and relaxes in time (we take $\mathcal Q(c)=-\rho_c c$).

The resulting dynamics exhibit a characteristic two-phase response. 
First, a transient increase in the deviation norm $E_S(t)=\|S(\cdot,t)-S^*\|_{L^2(U)}$ may occur as chemotactic drift concentrates mass near local signal maxima. 
Second, as the signal relaxes and the dissipative reaction terms dominate, the system returns toward the homogeneous equilibrium and $E_S(t)$ decays approximately exponentially.
This behavior is shown in \Cref{fig:norm_evolution} (red dashed curve), where the deviation rises briefly and then decreases.

Crucially, no persistent patterned attractor emerges: the transient focusing reflects short-time redistribution along an externally imposed landscape rather than de novo pattern formation from homogeneity.
These simulations therefore support the mechanism statement of \Cref{subsec:mechanism_statement} and are consistent with the unidirectional stability conclusions of \Cref{subsec:oneway}, namely that one-way sensing can amplify perturbations transiently but does not generate sustained diffusion-driven (Turing-type) patterns without closed-loop feedback.

\subsection{Bidirectional feedback: finite-band patterns vs aggregation (Regimes II--III)}
\label{subsec:numerical_bifeedback}

We now demonstrate the central distinction of the directionality--damping principle: sustained finite-wavelength structure is recovered only when a feedback loop closes the signal landscape.
In the unidirectional (open-loop) chemotaxis experiments, the cue evolves independently of $(S,R)$ and relaxes, so chemotaxis may transiently focus mass but cannot generate a persistent patterned attractor from a homogeneous equilibrium.
To create the minimal closed-loop mechanism, we introduce an auxiliary chemoattractant field $c$ that is produced by the susceptible population $S$ and decays in time.

We augment the unidirectional system by coupling $S$ to $c$ through a production--diffusion--decay equation, and we include taxis of $S$ along $\nabla c$:
\begin{align}
\label{eq:feedback_model_c}
\begin{cases}
\partial_t S
= d_S \Delta S
- \chi_S' \nabla \cdot (S \nabla c)
+ \lambda_S S\left(1 - \tfrac{S + R}{K}\right)
- \alpha S - \delta(I) S + \xi(1 - \phi(I)) R,
\\[4pt]
\partial_t c
= d_c \Delta c + \kappa_c S - \rho_c c,
\end{cases}
\end{align}
while keeping $(R,I,P,F_a)$ as in the base system \eqref{eq:system}.
Here $\kappa_c$ and $\rho_c$ are the production and decay rates of $c$, and $\chi_S'$ is the taxis sensitivity of $S$ to $c$.
This closes the positive feedback loop 
\[
S\to c\to \nabla c \to \text{ drift of } S.
\]
For the homogeneous equilibrium of \eqref{eq:feedback_model_c}, the signal component satisfies $c^*=(\kappa_c/\rho_c)\,S^*$.

\subsubsection{Regime II (well-posed finite-band patterns): flux saturation as a physical regularization}
To isolate finite-wavelength selection while preventing non-physical collapse at large focusing strength, we implement a flux-saturated (velocity-limited) taxis in the simulations underlying \Cref{fig:chemotaxis_attractant}.
Concretely, in \eqref{eq:feedback_model_c} we replace the linear taxis flux by
\begin{equation}
\label{eq:flux_saturation_c}
-\chi_S'\nabla\cdot(S\nabla c)
\quad\leadsto\quad
-\nabla\cdot\!\left(\chi_S' \, S\, \frac{\nabla c}{1+\alpha_c\lvert \nabla c\rvert}\right),
\end{equation}
with $\alpha_c=100$ controlling the saturation strength.
$\alpha_c=100$ is chosen to strongly limit drift velocities and prevent grid-scale collapse;
qualitative behavior is insensitive within an order of magnitude.
This form bounds the drift velocity as $\lvert \nabla c\rvert$ increases, suppressing high-frequency collapse while retaining the intermediate-wavelength destabilization.
With this regularization in place, the feedback-augmented system develops persistent finite-wavelength spatial heterogeneity; \Cref{fig:chemotaxis_attractant} shows the patterned $S$ field together with the chemoattractant $c$ and a representative cross-section indicating their phase alignment.

\begin{figure}[htbp]
\centering
\includegraphics[width=\textwidth]{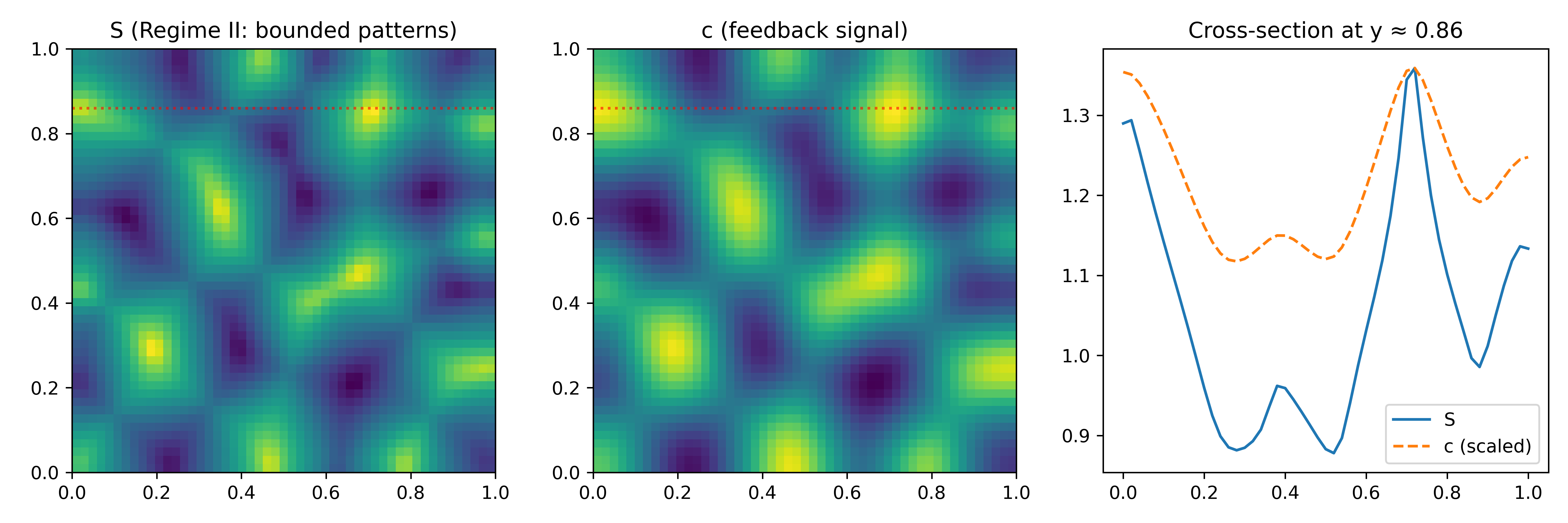}
\caption{\textbf{Restoration of patterning via bidirectional feedback.}
Left: spatial profile of $S$ showing finite-wavelength structures.
Middle: chemoattractant field $c$ generated by $S$ via \eqref{eq:feedback_model_c}.
Right: cross-section at $y=0.86$ showing phase alignment between $S$ and $c$.}
\label{fig:chemotaxis_attractant}
\end{figure}

\subsubsection{Regime III (aggregation/ill-posed tendency): removal of saturation}
In contrast, when the same feedback system is simulated without flux saturation (i.e., using the linear taxis flux in \eqref{eq:feedback_model_c}), the focusing drift can overwhelm diffusive smoothing and rapidly amplify small scales. 
Numerically, this manifests as overflow/NaN and breakdown rather than a smooth patterned attractor; see \Cref{fig:collapse}.
Accordingly, we interpret the breakdown as an aggregation tendency that necessitates a regularization mechanism (here: flux limitation), not as classical finite-wavelength pattern formation. 
In our parameter sweeps (with all other parameters fixed as in Table~\ref{tab:parameters}), we observed that the linear taxis sensitivity $\chi_S'$ admits a transition: below an $\mathcal{O}(10^{-2})$ level the simulations remain well-behaved and patterns (when present) stay bounded, whereas above this level the unregulated linear flux tends to trigger rapid small-scale amplification and numerical breakdown.

\begin{figure}[htbp]
\centering
\includegraphics[width=0.7\linewidth]{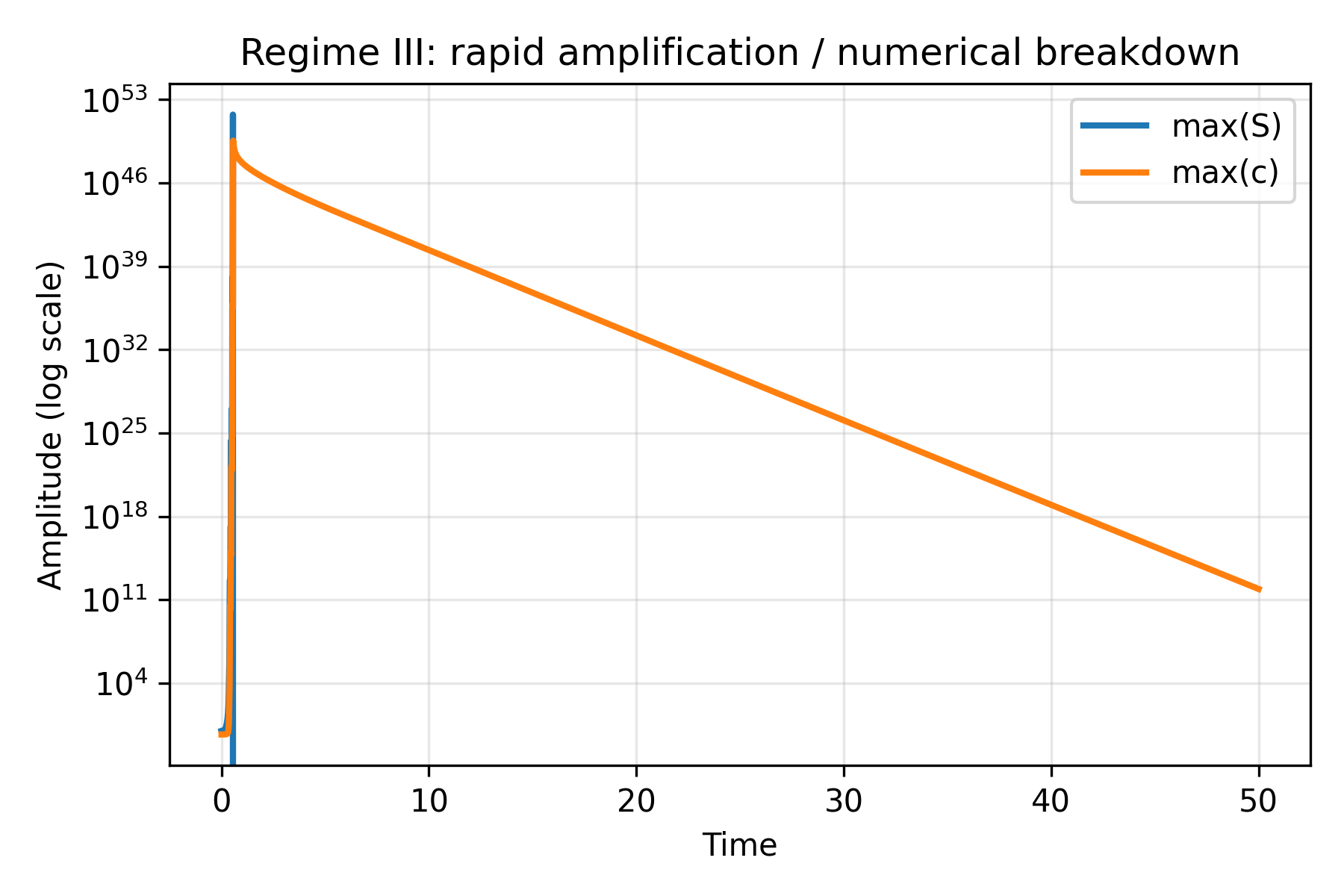}
\caption{\textbf{Numerical breakdown without flux limitation.}
In the bidirectional feedback setting \eqref{eq:feedback_model_c} with the linear taxis flux (no flux limitation), $\max S(t)$ and $\max c(t)$ grow rapidly and the simulation breaks down.
This behavior is consistent with aggregation/ill-posedness (Regime~III) and signals the need for regularization.}
\label{fig:collapse}
\end{figure}

\medskip
Taken together, the simulations in Regimes~I, I$'$, II, and III provide reproducible evidence for the paper's main message: under damping/relaxation, unidirectional coupling relaxes rather than generates persistent patterns, whereas bidirectional feedback can generate finite-wavelength structure in a well-posed regime, with flux regularization separating bounded patterning from aggregation-type breakdown.

\section{Discussion}
\label{sec:discussion}

We summarize the main message of this work and place it in the broader context of spatial instability mechanisms in hybrid PDE--ODE models. 
The proposed framework couples two proliferating tumor phenotypes $(S,R)$ to an exogenously administered therapeutic agent $I$ (single dose, open loop, no sustained source) and to a stromal switching module $(P,F_a)$ that is non-diffusing. 
All PDE components are posed with homogeneous Neumann boundary conditions, representing a closed tissue domain and ensuring compatibility with spatially homogeneous equilibria and with the long-time reduction developed in \Cref{sec:base_analysis}. 
Within this setting, the analysis isolates how the topology of directed transport couplings interacts with damping and signal relaxation to determine whether spatial heterogeneity is generated from homogeneity, or instead decays.

A first structural feature is the open-loop decay of the drug field. 
Because $I$ satisfies a linear parabolic equation with clearance, it relaxes exponentially after dosing, and therefore all $I$-dependent couplings in the remaining equations vanish asymptotically through $\delta(I)$ and $\phi(I)$; see \Cref{subsec:long_time_reduction}. 
Mathematically, this yields a clean separation between transient drug-mediated effects and the intrinsic long-time behavior selected by the population kinetics. 
Biologically, it formalizes a limitation of single-dose interventions in an open-loop setting: in the absence of sustained input, drug-mediated modulation alone cannot maintain persistent spatial structure, and any enduring heterogeneity must be supported by an internal generation mechanism.

At the baseline level (no chemotaxis), the model exhibits strong dissipation. 
Diffusion acts as a spatial smoothing mechanism under Neumann boundary conditions, while logistic competition and switching provide damping in the reaction terms. 
Consistently, we establish global well-posedness and invariant-region structure, and we show that diffusion-driven (Turing-type) destabilization does not occur for the damped reaction--diffusion core near the post-washout coexistence state; see \Cref{subsec:turing_stability}. 
This clarifies that the baseline hybrid coupling through $(P,F_a)$ does not by itself supply a PDE-level pattern generator. 
In particular, the locally conserved stromal pool in the baseline regime produces a neutral mode in the $(P,F_a)$ block, but it does not create spatial growth because the stromal variables do not diffuse at this level.

The principal mechanism contribution of the paper is the directionality--damping principle for chemotaxis-augmented extensions (\Cref{sec:chemotaxis}).
Under damping and signal relaxation, unidirectional sensing cannot generate patterns from homogeneity, whereas closed-loop feedback can, with strong parabolicity of $M$ separating finite-band selection from aggregation.
This modification produces a sharp classification into three regimes (\Cref{subsec:twoway_classification}): a stable regime with decay to homogeneity, a well-posed finite-band regime in which intermediate wavelengths destabilize while high frequencies remain damped, and an aggregation/ill-posed regime in which strong focusing defeats parabolic smoothing and leads to Hadamard-type high-frequency amplification.

A recurring theme across both the analysis and the numerical experiments is the role of strong parabolicity. 
In the bidirectional setting, the effective mobility matrix $M$ provides a mathematically precise indicator separating finite-wavelength selection from aggregation tendencies: when $\mathrm{sym}(M)$ remains positive definite, high frequencies are damped and any instability is necessarily confined to a bounded band; when this condition fails, high-frequency modes are no longer controlled and the unregularized continuum description can enter an ill-posed regime. 
From a modelling perspective, this emphasizes that strong positive feedback and steep taxis responses cannot be pushed arbitrarily far in continuum PDE descriptions without additional physics. 
Flux saturation and related mechanisms (e.g.\ volume filling) act as biologically motivated regularizers that restore well-posedness and permit bounded finite-wavelength structure as a stable outcome, as illustrated numerically in \Cref{sec:numerical}.

\subsection{Clinical interpretation and scope}
\label{subsec:clinical_interpretation}
We summarize what the mechanism classification does and does not say clinically.
The results are qualitative, mechanism-level statements about a reduced model;
they are not calibrated treatment protocols, and the parameters used are
illustrative.

\emph{(i) Persistent spatial structure as a resistant-reservoir signature.}
A persistent finite-wavelength pattern in which the resistant phenotype $R$ is
locally enriched corresponds to spatially localized resistant niches, the
reservoirs most strongly implicated in post-therapy relapse \citep{liu2025bidirectional}. In the present
framework such persistence cannot arise from single-dose forcing alone
(Theorem~\ref{thm:oneway}, Corollaries~\ref{cor:oneway_logis_damping}--\ref{cor:oneway_damping}); when the dynamics relax to
homogeneity, this mechanism predicts no self-organized reservoir. The
qualitative outcome---relaxation versus persistent banding---is thus itself a
prognostic indicator within the model.

\emph{(ii) The therapeutic lever is the feedback loop, not the dose.}
Because open-loop forcing cannot generate patterns from homogeneity while
closed-loop feedback can (\Cref{thm:twoway_classification}), the model reframes
the intervention target: to suppress spatial resistance niches one should
interrupt the endogenous feedback (for example anti-chemotactic strategies or
stromal normalization) rather than only intensifying dose or scheduling.

\emph{(iii) A qualitative safe/pattern-forming boundary.}
The chemotactic thresholds $\chi_S^{c},\chi_R^{c}$ in
\eqref{eq:chemotaxis_threshold} and the strong-parabolicity boundary of
$M=\mathcal D_{\mathrm{diff}}-\mathcal H(\mathbf u^*)$ give a qualitative dividing
line between safely damped microenvironmental modulation and modulation that
tips into niche formation (or, beyond it, aggregation). Together with the
composition ratio $\rho=\xi/\alpha$ (Remark~\ref{rmk:phy}), this links model
parameters to readouts accessible from spatial assays.

\emph{Scope and limitations.} The model captures one clinically relevant
constraint---single-dose, open-loop pharmacokinetics with washout---and
deliberately isolates the competition--switching--taxis mechanism. It does not
represent motile cancer-associated fibroblasts and their role in invasion
(Section~\ref{subsec:base_model}), the immune compartment, quantitative
pharmacokinetics/pharmacodynamics, or patient-specific parameters. Accordingly
we frame the conclusions as organizing principles for when spatial heterogeneity
can be generated, not as quantitative therapeutic recommendations.

\subsection{Extensions}
Several extensions are natural and remain mathematically tractable within the present framework. 
Repeated dosing or continuous infusion can be incorporated by adding an input term to the $I$ equation, which would remove the strict $I\to0$ long-time reduction while preserving the conceptual separation between open-loop forcing and closed-loop feedback. 
Allowing small diffusion in the stromal variables (or introducing spatial coupling through tissue remodeling) would remove exact local conservation of $P+F_a$ and may introduce additional slow modes interacting with the Neumann spectrum. 
Multiple chemoattractants or phenotype-dependent taxis can be handled within the mobility-matrix perspective, at the expense of higher-dimensional stability conditions. 
Finally, intrinsic (demographic) variability can be incorporated by adding stochastic terms that preserve nonnegativity in an It\^o sense, providing a route to quantify robustness of the regime classification under noise \citep{abundo1991stochastic,wang2025analysis}.

Overall, the results delineate a coherent picture: in a hybrid PDE--ODE setting with a single-dose open-loop drug, persistent spatial heterogeneity requires a genuine generation mechanism, and chemotactic feedback provides such a mechanism precisely when transport remains in a strongly parabolic (or appropriately regularized) regime.

\backmatter

\section*{Data Availability:}
Data sharing is not applicable to this article as no new data were created or analyzed in this study.

\section*{Funding:}
All authors of this article have confirmed that this research received no external funding.

\section*{Author Contributions}
All authors contributed to the study. All authors read and approved the final
manuscript.

\section*{Ethical statement}
The authors declare that the research presented in this manuscript is original and has not
been published elsewhere and is not under consideration by another journal. The study was conducted in
accordance with ethical principles and guidelines.

\section*{Declaration:}
All authors declare no competing interests.

\begin{appendices}
\section{Global \texorpdfstring{$L^\infty$}{L-infinity} bound for the unidirectional templates}\label{appsec:global_bound_unidirectional}
We expand the arguments in Remark~\ref{rmk:wp_extension}.

\begin{proposition}[\textbf{Global $L^\infty$ bound for the unidirectional templates}]
\label{prop:global_oneway}
Consider the unidirectional template \eqref{eq:oneway_logis_damping} (or
\eqref{eq:oneway}) under homogeneous Neumann boundary conditions, with $f_S,f_R$
satisfying Assumption~\ref{ass:reaction_growth} and nonnegative initial data
$(S_0,R_0)\in (L^\infty(U))^2$. Assume the decoupled signal equation
$\partial_t c=d_c\Delta c+\mathcal Q(c)$, $\partial_c\mathcal Q<0$, has
$\mathcal Q\in C^1$ and initial datum $c_0\in C^{2+\alpha}(\overline U)$ satisfying
the Neumann compatibility condition, so that (parabolic Schauder theory, and
independently of $(S,R)$)
\[
c\in C^{2+\alpha,\,1+\alpha/2}(\overline U\times[0,T]),
\qquad
\nabla c,\ \Delta c\in L^\infty(\overline U\times[0,T])\ \text{ for every }T>0 .
\]
Then $S,R\ge0$, and for every $T<T_{\max}$,
\[
\sup_{0\le t\le T}\bigl(\|S(t)\|_{L^\infty(U)}+\|R(t)\|_{L^\infty(U)}\bigr)\le C_T<\infty .
\]
Consequently $T_{\max}=\infty$ and the solution is global.
\end{proposition}

\begin{proof}
\emph{Step 1: Nonnegativity.} As in Theorem~\ref{thm:positivity}, consider the first
time and point at which a component, say $S$, attains the value $0$ as a spatial
minimum. There $\nabla S=0$, so the taxis contribution
$-\chi_S\nabla\!\cdot(S\nabla c)=-\chi_S\nabla S\!\cdot\!\nabla c-\chi_S S\,\Delta c$
vanishes (both terms carry a factor $\nabla S$ or $S$, each zero), and
$\Delta S\ge0$; the interior/boundary minimum argument of Theorem~\ref{thm:positivity}
applies verbatim because the drift $\chi_S\nabla c\in L^\infty$ is a lower-order
term. Hence $\partial_t S\ge f_S(0,R)\ge0$ by quasi-positivity, and $S$ cannot
become negative; symmetrically for $R$. Thus $S,R\ge0$ on $[0,T_{\max})$.

\emph{Step 2: Scalar differential inequalities for the sup-norms.} Set
$K:=\|\Delta c\|_{L^\infty(\overline U\times[0,T])}<\infty$ and
$a:=L+\max\{\chi_S,\chi_R\}\,K$. Expanding the taxis divergence,
$-\chi_S\nabla\!\cdot(S\nabla c)=-\chi_S\nabla c\!\cdot\!\nabla S-\chi_S(\Delta c)S$,
and using $S\ge0$ together with \eqref{eq:linear_growth_bound} and
$R(x,t)\le\|R(t)\|_\infty$,
\begin{equation}
\label{eq:scalar_ineq}
\partial_t S-d_S\Delta S+\chi_S\nabla c\!\cdot\!\nabla S
\;\le\;a\,S+L\bigl(1+\|R(t)\|_\infty\bigr)\qquad\text{in }U\times(0,T].
\end{equation}
The right-hand side is spatially constant. The constant-in-space function
$\bar S(t)$ solving the scalar ODE
$\bar S'(t)=a\,\bar S(t)+L(1+\|R(t)\|_\infty)$, $\bar S(0)=\|S_0\|_\infty$, is a
supersolution of the scalar linear parabolic operator
$\partial_t-d_S\Delta+\chi_S\nabla c\!\cdot\!\nabla$ (it has zero spatial
derivatives and satisfies \eqref{eq:scalar_ineq} with equality). Since this
operator has bounded coefficients and satisfies the (scalar) parabolic maximum
principle under Neumann conditions \citep{protter2012maximum,ladyzenskaia1968linear}, comparison
gives $S(x,t)\le\bar S(t)$, i.e.
\[
\frac{d^+}{dt}\|S(t)\|_\infty\le a\,\|S(t)\|_\infty+L\bigl(1+\|R(t)\|_\infty\bigr),
\]
and symmetrically
$\frac{d^+}{dt}\|R(t)\|_\infty\le a\,\|R(t)\|_\infty+L\bigl(1+\|S(t)\|_\infty\bigr)$,
where $\tfrac{d^+}{dt}$ is the upper Dini derivative.

\emph{Step 3: Gronwall.} Let $z(t):=\|S(t)\|_\infty+\|R(t)\|_\infty$. Adding the two
inequalities and using $\|S\|_\infty,\|R\|_\infty\le z$,
\[
\frac{d^+}{dt}z(t)\le (a+L)\,z(t)+2L .
\]
By Gronwall's inequality, $z(t)\le \bigl(z(0)+\tfrac{2L}{a+L}\bigr)e^{(a+L)t}$ for
$a+L>0$ (and $z(t)\le z(0)+2Lt$ if $a+L=0$), so
$\sup_{[0,T]}z\le C_T<\infty$ for every $T<T_{\max}$.

\emph{Step 4: Global existence.} Combining the bound on $(S,R)$ with the uniform
bounds for $(I,P,F_a)$ of Proposition~\ref{prop:IPA_bounds} and the $L^\infty$
blow-up alternative of Theorem~\ref{thm:global} excludes finite-time blow-up;
hence $T_{\max}=\infty$.
\end{proof}

\begin{corollary}[\textbf{Global bound for the concrete unidirectional model} \eqref{eq:oneway_damping}]
\label{cor:global_oneway_concrete}
For \eqref{eq:oneway_damping} the kinetics depend on the bounded quantities $I$ and
$F_a$ (Proposition~\ref{prop:IPA_bounds}), and Assumption~\ref{ass:reaction_growth}
holds with the explicit constant
$L=\max\{\lambda_S,\ \xi,\ \alpha,\ \lambda_R+\eta F_{a,\max}\}$: indeed
$\mathcal R_S\le\lambda_S S+\xi R$ and $\mathcal R_R\le\alpha S+(\lambda_R+\eta F_{a,\max})R$,
since the logistic terms are dissipative, $-\alpha S,-\delta(I)S\le0$,
$\xi(1-\phi)\le\xi$, and $\eta\phi(I)F_a\le\eta F_{a,\max}$; quasi-positivity holds as
$\mathcal R_S|_{S=0}=\xi(1-\phi)R\ge0$ and $\mathcal R_R|_{R=0}=\alpha S\ge0$. Hence
Proposition~\ref{prop:global_oneway} applies and \eqref{eq:oneway_damping} has
global classical solutions.
\end{corollary}

\section{Unidirectional tumor--drug--stroma chemotaxis: full dispersion relation}\label{app:proof_corollary_oneway_damping}

This appendix supplies the complete linearization and Neumann-mode dispersion relation for the unidirectional (open-loop) tumor--drug--stroma chemotaxis model used in \Cref{subsec:oneway}, thereby justifying Corollary~\ref{cor:oneway_damping}.  

\subsection{The unidirectional (open-loop) chemotaxis extension}
We consider the following unidirectional extension of the base hybrid model \eqref{eq:system}.  The tumor phenotypes $(S,R)$ undergo Keller--Segel-type drift along the gradient of a signal field $c(x,t)$, while $c$ evolves independently (no feedback from $(S,R)$).  The remaining components $(I,P,F_a)$ evolve as in \eqref{eq:system}.  Concretely,
\begin{align}
\label{eq:oneway_damping_final}
\begin{cases}
\partial_t S
= d_S\Delta S - \chi_S\nabla\!\cdot(S\nabla c)
+ \lambda_S S\Bigl(1-\dfrac{S+R}{K}\Bigr) - \alpha S - \delta(I)S + \xi\bigl[1-\phi(I)\bigr]R,\\[6pt]
\partial_t R
= d_R\Delta R - \chi_R\nabla\!\cdot(R\nabla c)
+ \lambda_R R\Bigl(1-\dfrac{S+R}{K}\Bigr) + \alpha S + \eta\,\phi(I)F_a R - \xi\bigl[1-\phi(I)\bigr]R,\\[6pt]
\partial_t I = d_I\Delta I - \gamma_I I,\\[6pt]
\partial_t P = -\theta \phi(I)P + \beta\bigl[1-\phi(I)\bigr]F_a,\\[6pt]
\partial_t F_a = \theta \phi(I)P - \beta\bigl[1-\phi(I)\bigr]F_a,\\[6pt]
\partial_t c = d_c \Delta c + \mathcal Q(c),
\qquad \mathcal Q'(c)<0.
\end{cases}
\end{align}
Homogeneous Neumann boundary conditions are imposed for the diffusing variables $(S,R,I,c)$:
\begin{equation}
\label{eq:neumann_bc_oneway_appendix}
\partial_{\mathrm n}S=\partial_{\mathrm n}R=\partial_{\mathrm n}I=\partial_{\mathrm n}c=0
\qquad\text{on }\partial U\times(0,\infty).
\end{equation}
In the baseline hybrid regime $(P,F_a)$ evolves pointwise in $x$. The signal relaxation condition $\mathcal Q'(c)<0$ encodes strict damping of the open-loop guidance field.

\subsection{Homogeneous equilibrium and Jacobian blocks}
Let
\[
\mathbf u^*=(S^*,R^*,I^*,P^*,F_a^*,c^*)
\]
be a spatially homogeneous equilibrium of \eqref{eq:oneway_damping_final}, i.e.
\[
\Delta S^*=\Delta R^*=\Delta I^*=\Delta c^*=0,\qquad
\mathcal Q(c^*)=0,
\]
and the reaction terms vanish at $(S^*,R^*,I^*,P^*,F_a^*)$.

Define, for convenience,
\[
\phi^*:=\phi(I^*),\qquad \delta^*:=\delta(I^*),\qquad \sigma^*:=\xi\bigl(1-\phi^*\bigr),\qquad \zeta^*:=\eta\,\phi^* F_a^*.
\]
The kinetic $(S,R)$ reaction map at fixed $(I^*,F_a^*)$ is
\[
F_S(S,R)=\lambda_S S\Bigl(1-\frac{S+R}{K}\Bigr)-\alpha S-\delta^* S+\sigma^* R,
\qquad
F_R(S,R)=\lambda_R R\Bigl(1-\frac{S+R}{K}\Bigr)+\alpha S+\zeta^* R-\sigma^* R.
\]
Its Jacobian at $(S^*,R^*)$ is
\begin{equation}
\label{eq:JF_SR_appendix}
\mathcal J_{SR}:=
\begin{pmatrix}
a & b\\
c & d
\end{pmatrix},
\end{equation}
with the explicit entries
\begin{align}
\label{eq:abcd_appendix}
\begin{aligned}
a &= \partial_S F_S(S^*,R^*)= \lambda_S\Bigl(1-\frac{2S^*+R^*}{K}\Bigr)-\alpha-\delta^*,\\[4pt]
b &= \partial_R F_S(S^*,R^*)= -\frac{\lambda_S}{K}S^*+\sigma^*,\\[4pt]
c &= \partial_S F_R(S^*,R^*)= \alpha-\frac{\lambda_R}{K}R^*,\\[4pt]
d &= \partial_R F_R(S^*,R^*)= \lambda_R\Bigl(1-\frac{S^*+2R^*}{K}\Bigr)+\zeta^*-\sigma^*.
\end{aligned}
\end{align}
(The dependence of $\phi(\cdot)$ and $\delta(\cdot)$ on $I$ produces additional forcing terms proportional to the $I$-perturbation; these appear below and are shown not to affect eigenvalues because the $I$-equation is open-loop.)

The $c$-block linearization is scalar and strictly stable:
\begin{equation}
\label{eq:c_block_appendix}
\partial_t \tilde c = d_c \Delta \tilde c + \mathcal Q'(c^*)\,\tilde c,
\qquad \mathcal Q'(c^*)<0.
\end{equation}
The drug block is also strictly stable:
\begin{equation}
\label{eq:I_block_appendix}
\partial_t \tilde I = d_I \Delta \tilde I - \gamma_I \tilde I.
\end{equation}
Finally, in the baseline hybrid regime, the stromal switching block is pointwise linear:
\begin{equation}
\label{eq:PA_block_appendix}
\partial_t
\begin{pmatrix}\tilde P\\ \tilde F_a\end{pmatrix}
=
B(I^*)
\begin{pmatrix}\tilde P\\ \tilde F_a\end{pmatrix}
\ +\ 
\tilde I\,
B_I(I^*)
\begin{pmatrix}P^*\\ F_a^*\end{pmatrix},
\qquad
B(I^*)=
\begin{pmatrix}
-\theta\phi^* & \beta(1-\phi^*)\\
\theta\phi^* & -\beta(1-\phi^*)
\end{pmatrix},
\end{equation}
where $B_I(I^*)$ denotes the derivative of $B(I)$ with respect to $I$ at $I=I^*$ (it is proportional to $\phi'(I^*)$).  Note that $B(I^*)$ has eigenvalues
\begin{equation}
\label{eq:PA_eigs_appendix}
\lambda_{PA,1}=0,
\qquad
\lambda_{PA,2}=-(\theta\phi^*+\beta(1-\phi^*))\le 0,
\end{equation}
so the $(P,F_a)$ block contributes no positive spectrum.  The eigenvalue $0$ corresponds to the locally conserved pool $P+F_a$ in the baseline regime.

\subsection{Linearization of the chemotaxis terms}
Write perturbations
\[
S=S^*+\tilde S,\quad R=R^*+\tilde R,\quad I=I^*+\tilde I,\quad P=P^*+\tilde P,\quad F_a=F_a^*+\tilde F_a,\quad c=c^*+\tilde c.
\]
Since $c^*$ is spatially constant, $\nabla c^*=0$.  The taxis fluxes linearize to
\[
-\chi_S\nabla\!\cdot(S\nabla c) \;=\; -\chi_S\nabla\!\cdot\bigl((S^*+\tilde S)\nabla(c^*+\tilde c)\bigr)
\;=\; -\chi_S S^* \Delta \tilde c \;+\; \text{(higher order)},
\]
and similarly
\[
-\chi_R\nabla\!\cdot(R\nabla c)= -\chi_R R^* \Delta \tilde c + \text{(higher order)}.
\]
Crucially, the taxis contributes only a forcing from $\tilde c$ into the $(\tilde S,\tilde R)$ equations; it does not modify the $(\tilde S,\tilde R)$ Jacobian with respect to $(\tilde S,\tilde R)$ at a homogeneous equilibrium.

\subsection{Neumann mode decomposition and the dispersion matrices}
Let $\{-\Delta w_k=\mu_k w_k,\ \partial_{\mathrm n}w_k=0\}_{k\ge 0}$ be the Neumann Laplacian eigenpairs on $U$, with $\mu_k\ge 0$ and $\{w_k\}$ an orthonormal basis of $L^2(U)$.
Expand each perturbation in this basis:
\[
\tilde S=\sum_k s_k(t)w_k(x),\quad
\tilde R=\sum_k r_k(t)w_k(x),\quad
\tilde I=\sum_k i_k(t)w_k(x),\quad
\tilde c=\sum_k \hat c_k(t)w_k(x),
\]
and similarly for $\tilde P,\tilde F_a$.

For each mode $\mu=\mu_k$, the $(\tilde c,\tilde I)$ coefficients satisfy the decoupled ODEs
\begin{equation}
\label{eq:cI_modes_appendix}
\hat c_k' = \bigl(\mathcal Q'(c^*)-d_c\mu\bigr)\hat c_k,
\qquad
i_k' = -(\gamma_I+d_I\mu)\,i_k.
\end{equation}
Hence
\[
\mathcal Q'(c^*)-d_c\mu<0\ \ \forall \mu\ge 0,
\qquad
-(\gamma_I+d_I\mu)<0\ \ \forall \mu\ge 0,
\]
so neither $c$ nor $I$ can contribute unstable eigenvalues.

The $(\tilde S,\tilde R)$ mode amplitudes satisfy
\begin{equation}
\label{eq:SR_modes_appendix}
\begin{pmatrix}s_k\\ r_k\end{pmatrix}'=
\underbrace{\Bigl(\mathcal J_{SR}-\mu\,\mathrm{diag}(d_S,d_R)\Bigr)}_{=:\,\mathcal A_{SR}(\mu)}
\begin{pmatrix}s_k\\ r_k\end{pmatrix}
\ +\ 
\mu
\begin{pmatrix}
\chi_S S^*\\
\chi_R R^*
\end{pmatrix}\hat c_k
\ +\ 
\begin{pmatrix}
\Gamma_S\\ \Gamma_R
\end{pmatrix} i_k,
\end{equation}
where $\Gamma_S,\Gamma_R$ are constants determined by derivatives of $\delta(I)$ and $\phi(I)$ at $I^*$ (they are proportional to $\delta'(I^*)$ and $\phi'(I^*)$).  The important point is that $\hat c_k$ and $i_k$ enter \eqref{eq:SR_modes_appendix} only as inhomogeneous forcing: the homogeneous part is governed by the $2\times 2$ dispersion matrix
\begin{equation}
\label{eq:ASR_def_appendix}
\mathcal A_{SR}(\mu)=\mathcal J_{SR}-\mu
\begin{pmatrix}
d_S & 0\\
0 & d_R
\end{pmatrix}.
\end{equation}

Likewise, the $(\tilde P,\tilde F_a)$ block is driven by $i_k$ but has homogeneous dynamics generated by $B(I^*)$ in \eqref{eq:PA_block_appendix}, with eigenvalues \eqref{eq:PA_eigs_appendix}.  Therefore, after ordering the mode variables as
\[
(\hat c_k,\, i_k,\, \tilde p_k,\, \tilde a_k,\, s_k,\, r_k),
\]
the full mode-wise linear system is block lower triangular: $(\hat c_k,i_k)$ evolve independently, $(\tilde p_k,\tilde a_k)$ are driven by $i_k$ only, and $(s_k,r_k)$ are driven by $(\hat c_k,i_k)$ only.  Consequently, the spectrum for each $\mu$ is the union of the spectra of the diagonal blocks:
\begin{equation}
\label{eq:spectrum_union_appendix}
\sigma(\mu)=\Bigl\{\mathcal Q'(c^*)-d_c\mu,\ -(\gamma_I+d_I\mu),\ 0,\ -(\theta\phi^*+\beta(1-\phi^*)),\ \lambda_\pm(\mu)\Bigr\},
\end{equation}
where $\lambda_\pm(\mu)$ are the two eigenvalues of $\mathcal A_{SR}(\mu)$.

\subsection{Explicit trace and determinant of \texorpdfstring{$\mathcal A_{SR}(\mu)$}{A\_SR(mu)}}
From \eqref{eq:ASR_def_appendix} and \eqref{eq:JF_SR_appendix},
\begin{equation}
\label{eq:trace_det_appendix}
\mathrm{tr}\bigl(\mathcal A_{SR}(\mu)\bigr)=(a+d)-\mu(d_S+d_R),
\qquad
\det\bigl(\mathcal A_{SR}(\mu)\bigr)=(a-\mu d_S)(d-\mu d_R)-bc.
\end{equation}
Equivalently, writing the determinant as a quadratic polynomial in $\mu$,
\begin{equation}
\label{eq:det_quadratic_appendix}
\det\bigl(\mathcal A_{SR}(\mu)\bigr)
=d_S d_R\,\mu^2-(d\,d_S+a\,d_R)\mu+(ad-bc).
\end{equation}
Hence the mode-wise characteristic polynomial for the $(S,R)$ block is
\begin{equation}
\label{eq:charpoly_SR_appendix}
\lambda^2-\mathrm{tr}\bigl(\mathcal A_{SR}(\mu)\bigr)\lambda+\det\bigl(\mathcal A_{SR}(\mu)\bigr)=0,
\end{equation}
and the Routh--Hurwitz conditions for $\Re\lambda_\pm(\mu)<0$ are exactly
\[
\mathrm{tr}\bigl(\mathcal A_{SR}(\mu)\bigr)<0,
\qquad
\det\bigl(\mathcal A_{SR}(\mu)\bigr)>0.
\]

\subsection{Why the \texorpdfstring{$c$}{c}-block cannot create unstable eigenvalues}
The $c$-equation is autonomous and strictly dissipative at the linear level: by \eqref{eq:cI_modes_appendix}, for every Neumann mode $\mu\ge 0$,
\[
\hat c_k'=\bigl(\mathcal Q'(c^*)-d_c\mu\bigr)\hat c_k,
\qquad \mathcal Q'(c^*)-d_c\mu<0.
\]
Thus the $c$-block contributes only strictly negative eigenvalues.  Moreover, because $c$ receives no feedback from $(S,R)$ in \eqref{eq:oneway_damping_final}, its block is diagonal in the mode-wise ordering and cannot be shifted by coupling.  The chemotaxis terms appear only as forcing from $\hat c_k$ into $(s_k,r_k)$ in \eqref{eq:SR_modes_appendix}, which does not affect the eigenvalues (cf.\ the block-triangular spectrum union \eqref{eq:spectrum_union_appendix}).

\subsection{Completion of Corollary~\ref{cor:oneway_damping}}
Assume the homogeneous equilibrium is kinetically stable in the sense that the $(S,R)$ Jacobian satisfies
\begin{equation}
\label{eq:kinetic_stability_appendix}
a+d<0,\qquad ad-bc>0.
\end{equation}
Then $\mathrm{tr}(\mathcal A_{SR}(\mu))<0$ for all $\mu>0$ because $\mu(d_S+d_R)>0$.  To ensure stability of every non-constant Neumann mode it suffices to guarantee
\[
\det(\mathcal A_{SR}(\mu))>0\quad\forall \mu\ge 0,
\]
which holds under either of the two standard sufficient conditions (exactly as in Theorem~\ref{thm:oneway}):
\[
d\,d_S+a\,d_R<0,
\qquad \text{or}\qquad
4(ad-bc)\,d_Sd_R>(d\,d_S+a\,d_R)^2.
\]
Under these conditions, the eigenvalues $\lambda_\pm(\mu)$ of the $(S,R)$ dispersion matrix satisfy $\Re\lambda_\pm(\mu)<0$ for all $\mu>0$.

Finally, the remaining diagonal blocks in \eqref{eq:spectrum_union_appendix} contribute only nonpositive spectrum: the drug eigenvalue is $-(\gamma_I+d_I\mu)<0$, the signal eigenvalue is $\mathcal Q'(c^*)-d_c\mu<0$, and the stromal block contributes $\{0,-(\theta\phi^*+\beta(1-\phi^*))\}$ with the zero eigenvalue corresponding to the locally conserved pool in the baseline hybrid regime. Therefore, no diffusion-driven destabilization can arise from unidirectional taxis: open-loop sensing does not generate Turing-type patterns from a homogeneous equilibrium, completing the promised derivation.

\subsection{Specialization to the single-dose washout equilibrium}
In the single-dose setting emphasized in the main text, the long-time homogeneous equilibrium satisfies $I^*=0$ (drug washout), hence $\phi^*=\phi(0)$ and $\delta^*=\delta(0)=0$; in typical choices one has $\phi(0)=0$, so $\sigma^*=\xi$ and $\zeta^*=0$.  In that case the Jacobian entries \eqref{eq:abcd_appendix} reduce to those of the reduced $(S,R)$ subsystem \eqref{eq:dynamical}.  The dispersion matrix \eqref{eq:ASR_def_appendix} then matches the unidirectional template in \Cref{subsec:oneway}, and the above conditions reduce exactly to the criteria stated in Theorem~\ref{thm:oneway}.

\end{appendices}

\bibliography{reference}

\end{document}